\documentclass[11pt,reqno]{amsart}

\usepackage[margin=1in]{geometry}
\usepackage{microtype}
\usepackage{setspace}

\usepackage{amsmath}
\usepackage{amssymb}
\usepackage{amsthm}
\usepackage{mathtools}
\usepackage{bm}

\usepackage{enumitem}
\usepackage{xcolor}

\usepackage[
    colorlinks=true,
    linkcolor=blue,
    citecolor=blue,
    urlcolor=blue
]{hyperref}

\usepackage[nameinlink,noabbrev]{cleveref}

\theoremstyle{plain}

\newtheorem{theorem}{Theorem}[section]
\newtheorem{proposition}[theorem]{Proposition}
\newtheorem{lemma}[theorem]{Lemma}
\newtheorem{corollary}[theorem]{Corollary}

\theoremstyle{definition}

\newtheorem{definition}[theorem]{Definition}
\newtheorem{example}[theorem]{Example}

\theoremstyle{remark}

\newtheorem{remark}[theorem]{Remark}

\newcommand{\He}{\operatorname{He}}

\title{Classification of Collisions of Twisted Foulkes Character Polynomials}

\author{Aparna Upadhyay
}

\address{Department of Mathematics \& Statistics\\ University of South Alabama \\
411 N University Blvd\\Mobile, AL~36688, USA}

\email{aupadhyay@southalabama.edu}

\date{August 14, 2026}

\begin{document}

\maketitle


\begin{abstract}
The twisted Foulkes character polynomial is an algebraically defined polynomial attached to an integer partition. We determine precisely how much combinatorial information this polynomial encodes by completely classifying all pairs of partitions that give rise to the same polynomial. Our main result shows that equality of twisted Foulkes character polynomials admits a purely combinatorial characterization in terms of two explicit local operations on partitions.
\end{abstract}


\section{Introduction}
\label{sec:introduction}

Representation theory frequently gives rise to polynomial constructions encoding discrete combinatorial data. A fundamental question is to determine how much of the underlying combinatorial object can be recovered from such a polynomial. In this paper, we study the twisted Foulkes character polynomial attached to an integer partition, arising naturally from the twisted Foulkes characters of the symmetric group, and completely determine the fibers of the resulting map.

Let \(H^{(2^m)}\) denote the Foulkes module of the symmetric group \(S_{2m}\), obtained by inducing the trivial module from the imprimitive wreath product \(S_2\wr S_m\) to \(S_{2m}\). Given nonnegative integers \(m\) and \(k\), define \[
H^{(2^m;k)}
:=
\left(H^{(2^m)}\boxtimes\operatorname{sgn}_{S_k}\right)
\uparrow^{S_{2m+k}}_{S_{2m}\times S_k},
\]
where \(\boxtimes\) denotes the outer tensor product. These modules are called the \emph{twisted Foulkes modules} and generalize the ordinary Foulkes modules. If \(\chi^{(2^m;k)}\) denotes the ordinary character of
\(H^{(2^m;k)}\), then for an element \(g\in S_n\) of cycle type \(\lambda=(1^{s_1},\ldots,n^{s_n})\), where
\(2m+k=n\), we know from \cite{Hall-U-Twisted} that
\[
\sum_{m=0}^{\lfloor n/2\rfloor}
\chi^{(2^m;k)}(g)x^k
=
\prod_{r=1}^{\lfloor n/2\rfloor}
\He_{s_{2r}}^{[-2r]}(1-x^{2r})
\He_{s_{2r-1}}^{[-(2r-1)]}(x^{2r-1}).
\]

Throughout this paper, we denote this polynomial by \(P_\lambda(x)\) and refer to it as the \emph{twisted Foulkes character polynomial}. Besides encoding all corresponding twisted Foulkes character values as its coefficients, it captures interesting representation-theoretic information about the symmetric group. 
For example, if $g\in S_n$ has cycle type $\lambda$, then $P_\lambda(1)$ is the value at $g$ of the total character, namely the sum of all irreducible characters of $S_n$, \cite[Corollary 4.10]{Hall-U-Twisted}. Although its definition is representation-theoretic, the twisted Foulkes character polynomial encodes remarkably rich combinatorial information about the underlying partition. One conceptual aspect of this construction is that it allows us to transform the discrete cycle structure of a conjugacy class into a finite configuration of roots in the complex plane. In this sense, the polynomial provides a geometric realization of the underlying partition.

Our main result shows that this realization is remarkably faithful. We introduce two explicit local operations on partitions (Definition \ref{def:two-operations}) and prove that they generate precisely the equivalence relation determined by equality of twisted Foulkes character polynomials. We obtain
\[
P_\lambda(x)=P_\mu(x) \Longleftrightarrow \lambda\sim\mu,
\]
where \(\sim\) is the equivalence relation defined in Definition \ref{def:equivalence-relation}. Thus every
collision admits a simple combinatorial description, and the polynomial determines the underlying partition up to these two elementary local ambiguities.

A key ingredient in the proof is the role played by generalized Hermite polynomials. Their structural and root-theoretic properties make it possible to recover the underlying partition data and thereby classify all polynomial collisions. This places generalized Hermite polynomials in a new algebraic-combinatorial setting and suggests new directions for the study of representation-theoretic polynomial invariants. Along the way, we develop several techniques that may be of independent interest, including quotient reduction, reciprocal core
factorizations, and logarithmic linearization. These methods provide a systematic framework for recovering partition data from multiplicative polynomial structures and may prove useful in the study of other
algebraically defined polynomials attached to combinatorial objects.

The paper is organized as follows. Section~\ref{sec:partition-polynomials} introduces the partition polynomials and establishes their basic properties. In Section~\ref{sec:equivalence} we define the equivalence relation on partitions and show that it preserves the associated polynomial. Section~\ref{sec:common-factor-rigidity} develops the rigidity theory needed for the classification and reduces the problem to two cases. Section~\ref{sec:equal-even-weight} treats the case in which the even weights agree, while Sections~\ref{sec:unequal-even-weight} and~\ref{sec:exceptional-case} develop the analytic machinery and complete the proof in the exceptional case of unequal even weights. We conclude in Section~\ref{sec:selected-roots} by recording several consequences concerning the roots of the twisted Foulkes character polynomial.

\section{Hermite Polynomials Associated to Partitions}
\label{sec:partition-polynomials}

In this section we introduce the Hermite polynomials associated to
integer partitions that form the central objects of this paper. We begin by describing the basic algebraic framework for the twisted Foulkes character polynomials, including a parity decomposition that underlies the classification. Throughout the remainder of the paper, these polynomials provide the bridge between the combinatorics of partitions and the analytic properties of Hermite polynomials.

\subsection{Ordinary and Generalized Hermite Polynomials}

We begin by recalling the probabilists' Hermite polynomials and introducing a one-parameter family of generalized Hermite polynomials that will play a central role throughout this paper.

\begin{definition}
For each nonnegative integer \(n\), the \emph{probabilists' Hermite
polynomial} is defined by
\[
\He_n(x)
=
\sum_{j=0}^{\lfloor n/2\rfloor}
(-1)^j
\binom{n}{2j}
(2j-1)!!
x^{\,n-2j},
\]
where, by convention,
\[
(-1)!!=1.
\]
\end{definition}

\begin{definition}\cite[pp. 87-89]{roman1984umbral}
Let \(v\in \mathbb{C}\). The \emph{generalized Hermite polynomial} of degree
\(n\) and parameter \(v\) is defined by
\[
\He_n^{[-v]}(x)
=
\sum_{j=0}^{\lfloor n/2\rfloor}
\binom{n}{2j}
(2j-1)!!
v^j
x^{\,n-2j}.
\]
When \(v=1\), these polynomials are obtained from the probabilists' Hermite polynomials by reversing the sign of the variance parameter, which explains the notation \(\He_n^{[-v]}(x)\).
\end{definition}

The following identity relates generalized Hermite polynomials to the ordinary Hermite polynomials.

\begin{theorem}[Imaginary rescaling formula]
\label{thm:imaginary-rescaling}
For every \(v>0\),
\[
\He_n^{[-v]}(x)
=
(i\sqrt{v})^n
\He_n\!\left(\frac{x}{i\sqrt{v}}\right).
\]
\end{theorem}

\begin{proof}
Substituting \(x/(i\sqrt{v})\) into the defining formula for
\(\He_n(x)\) gives
\[
\begin{aligned}
(i\sqrt{v})^n
\He_n\!\left(\frac{x}{i\sqrt{v}}\right)
&=
(i\sqrt{v})^n
\sum_{j=0}^{\lfloor n/2\rfloor}
(-1)^j
\binom{n}{2j}
(2j-1)!!
\left(\frac{x}{i\sqrt{v}}\right)^{n-2j}\\
&=
\sum_{j=0}^{\lfloor n/2\rfloor}
(-1)^j
\binom{n}{2j}
(2j-1)!!
(i\sqrt{v})^{2j}
x^{\,n-2j}.
\end{aligned}
\]
Since
\[
(i\sqrt{v})^{2j}
=
(-1)^jv^j,
\]
the two factors of \((-1)^j\) cancel, giving
\[
\sum_{j=0}^{\lfloor n/2\rfloor}
\binom{n}{2j}
(2j-1)!!
v^j
x^{\,n-2j},
\]
which is precisely \(\He_n^{[-v]}(x)\).
\end{proof}

As an immediate consequence, the roots of the generalized Hermite polynomials inherit the simplicity and reality properties of the ordinary Hermite polynomials.

\begin{corollary}
\label{cor:generalized-hermite-roots}
Let \(v>0\). Every root of $\He_n^{[-v]}(x)$ is simple and purely imaginary.
\end{corollary}

\begin{proof}
If \(x\) is a root of \(\He_n^{[-v]}(x)\), then by
Theorem~\ref{thm:imaginary-rescaling},
\[
\He_n\!\left(\frac{x}{i\sqrt{v}}\right)=0.
\]
Since every root of the probabilists' Hermite polynomial is real and simple, it follows that
\[
\frac{x}{i\sqrt{v}}
\]
is real. Hence every root of \(\He_n^{[-v]}(x)\) is purely imaginary. Simplicity is preserved under the linear change of variables
\[
x\mapsto\frac{x}{i\sqrt{v}},
\]
so every root is simple.
\end{proof}

Theorem \ref{thm:imaginary-rescaling} enables us to transfer many properties of the ordinary Hermite polynomials to the generalized Hermite polynomials. In particular, the location of their roots will play a
fundamental role in the rigidity arguments developed later in the paper.

\subsection{Partition Polynomials}

We associate a polynomial to each integer partition by assigning a generalized Hermite polynomial to every distinct part size. The construction naturally separates according to the parity of the parts, a feature that plays a fundamental role throughout the paper.

Let
\[
\lambda
=
\left(1^{s_1},2^{s_2},3^{s_3},\cdots,n^{s_n}\right)
\]
be a partition of a positive integer \(n\), where \(s_d=s_d(\lambda)\)
denotes the multiplicity of the part \(d\).

For every positive integer \(d\) and every nonnegative integer \(m\),
define
\[
F_{d,m}(x)
=
\He_m^{[-d]}(1-x^d),
\]
and
\[
G_{d,m}(x)
=
\He_m^{[-d]}(x^d).
\]

The polynomial \(F_{d,m}(x)\) will be associated with even parts $d$, whereas \(G_{d,m}(x)\) will be associated with odd parts $d$. Note that the degree of these polynomials is $dm$. The leading coefficient of the polynomial $F_{d,m}(x)$ is $(-1)^{m}$, and the polynomial \(G_{d,m}(x)\) is monic.

\begin{definition}
Let
\(
\lambda
=
\left(1^{s_1},2^{s_2},3^{s_3},\cdots,n^{s_n}\right)
\)
be a partition. The \emph{partition polynomial} associated to
\(\lambda\) is defined by
\[
P_\lambda(x)
=
\prod_{\substack{d\ge1\\ d\text{ even}}}
F_{d,s_d(\lambda)}(x)
\;
\prod_{\substack{d\ge1\\ d\text{ odd}}}
G_{d,s_d(\lambda)}(x).
\]
\end{definition}

The first few building blocks are
\[
\begin{aligned}
F_{d,0}(x)
&=1,\\
F_{d,1}(x)
&=1-x^d,\\
F_{d,2}(x)
&=(1-x^d)^2+d,
\end{aligned}
\]
and
\[
\begin{aligned}
G_{d,0}(x)
&=1,\\
G_{d,1}(x)
&=x^d,\\
G_{d,2}(x)
&=x^{2d}+d.
\end{aligned}
\]

\begin{example}
Let
\[
\lambda=(6,4,4,3,1,1).
\]
Then
\[
\begin{aligned}
P_\lambda(x)
&=
F_{6,1}(x)
F_{4,2}(x)
G_{3,1}(x)
G_{1,2}(x)\\
&=
\He^{[-6]}_1(1-x^6)\,
\He^{[-4]}_2(1-x^4)\,
\He^{[-3]}_1(x^3)\,
\He^{[-1]}_2(x).
\end{aligned}
\]
\end{example}

The product defining \(P_\lambda(x)\) is finite since only finitely many multiplicities \(s_d(\lambda)\) are nonzero. Moreover, the decomposition into even and odd factors is intrinsic to the construction and will allow the two types of factors to be analyzed independently in the subsequent sections.

\subsection{Even and Odd Decomposition}

The definition of the partition polynomial naturally separates into contributions coming from the even parts and the odd parts of the partition. Since these two families of factors exhibit markedly different algebraic behaviour, we shall study them independently throughout the paper.

\begin{definition}
Let
\(
\lambda
=
\left(1^{s_1},2^{s_2},3^{s_3},\cdots,n^{s_n}\right)
\)
be a partition. The \emph{even polynomial} associated to \(\lambda\) is
\[
E_\lambda(x)
=
\prod_{\substack{d\ge1\\ d\textnormal{ even}}}
F_{d,s_d(\lambda)}(x),
\]
and the \emph{odd polynomial} associated to \(\lambda\) is
\[
O_\lambda(x)
=
\prod_{\substack{d\ge1\\ d\textnormal{ odd}}}
G_{d,s_d(\lambda)}(x).
\]
\end{definition}

We also define the \emph{even weight} and \emph{odd weight} of a
partition by
\[
e_\lambda
=
\sum_{\substack{d\ge1\\ d\textnormal{ even}}}
d\,s_d(\lambda),
\]
and
\[
o_\lambda
=
\sum_{\substack{d\ge1\\ d\textnormal{ odd}}}
d\,s_d(\lambda).
\]
Thus
\(
|\lambda|
=
e_\lambda+o_\lambda
\), \(\deg E_\lambda(x)=e_\lambda\), \(\deg O_\lambda(x)=o_\lambda\) and \(\deg P_\lambda(x)=|\lambda|\).
The decomposition
\[
P_\lambda(x)
=
E_\lambda(x)\,O_\lambda(x)
\]
reflects the separation of the partition into its even and odd parts. This decomposition will play a central role throughout the paper. In the subsequent sections we shall investigate the algebraic properties
of the even and odd factors separately.


\section{Equivalence of Partitions}
\label{sec:equivalence}

In this section we introduce two elementary operations on integer partitions. These operations generate an equivalence relation on the set of all partitions and preserve the associated partition polynomials. The main result of this paper will show that these are, in fact, the only operations that preserve the partition polynomial.

\begin{definition} \label{def:two-operations}
The following two elementary operations may be performed on a partition.
\begin{itemize}
    \item \textit{Operation 1:} Suppose that a partition $\lambda$ contains distinct odd parts
\(
a_1,\ldots,a_r,
\)
each occurring with multiplicity exactly one, and does not contain the
distinct odd parts
\(
b_1,\ldots,b_s
\)
such that
\[
\sum_{i=1}^r a_i
=
\sum_{j=1}^s b_j.
\]
The partition $\mu$ obtained by removing the parts
\[
a_1,\ldots,a_r
\]
and inserting the parts
\[
b_1,\ldots,b_s
\]
is said to be obtained by \emph{Operation~1}.

\item  \textit{Operation 2:} Suppose that a partition $\lambda$ contains exactly one part of size \(4\), no parts of size \(2\), and no parts of size \(1\). Replacing the unique
part \(4\) by the parts
\[
(2,1,1)
\]
to obtain a partition $\mu$ is called \emph{Operation~2}.
\end{itemize}
\end{definition}

\begin{definition}\label{def:equivalence-relation}
Two partitions $\lambda$ and $\mu$ are said to be equivalent if they belong to the smallest equivalence relation generated by Operations~1 and~2. In this case we write
\[
\lambda\sim\mu.
\]
\end{definition}

The following proposition shows that each elementary operation preserves the associated partition polynomial.

\begin{proposition}
\label{prop:operations-preserve}
Operations~1 and~2 each preserve the partition polynomial.
\end{proposition}

\begin{proof}
Suppose first that Operation~1 replaces the distinct odd parts
\[
a_1,\ldots,a_r
\]
by the distinct odd parts
\[
b_1,\ldots,b_s,
\]
where
\[
\sum_{i=1}^r a_i
=
\sum_{j=1}^s b_j.
\]
Since every removed part has multiplicity one,
\[
G_{d,1}(x)=x^d.
\]
Therefore the contribution of the removed parts to the partition
polynomial is
\[
x^{a_1}\cdots x^{a_r}
=
x^{a_1+\cdots+a_r},
\]
while the contribution of the inserted parts is
\[
x^{b_1}\cdots x^{b_s}
=
x^{b_1+\cdots+b_s}.
\]
These are equal because the sums of the removed and inserted parts are equal. All remaining Hermite factors are unchanged, so the partition polynomial is preserved.

Now suppose that Operation~2 replaces the unique part \(4\) by the parts \((2,1,1)\). The corresponding Hermite factor
\[
\He^{[-4]}_1(1-x^4)
=
1-x^4
\]
is replaced by
\[
\He^{[-2]}_1(1-x^2)
\,
\He^{[-1]}_2(x)
=
(1-x^2)(x^2+1)
=
1-x^4.
\]
Again all remaining factors are unchanged, so the partition polynomial
is preserved.
\end{proof}

\begin{corollary}
\label{cor:equivalence-preserves}
If
\(
\lambda\sim\mu,
\)
then
\(
P_\lambda(x)=P_\mu(x).
\)
\end{corollary}

\begin{proof}
The equivalence relation is generated by Operations~1 and~2, each of which preserves the partition polynomial by Proposition~\ref{prop:operations-preserve}.
\end{proof}

%

\section{Rigidity of Common Factors}
\label{sec:common-factor-rigidity}

The main objective of this section is to determine the possible common factors of the building blocks introduced in Section~\ref{sec:partition-polynomials}. We begin by describing the geometry of their roots. This viewpoint naturally leads to a tangency criterion, from which we deduce a strong rigidity theorem for common irreducible factors. We will conclude by showing that every discrepancy between the even and odd factors of two equal partition polynomials is necessarily a power of the polynomial \(x^2+1\).

\subsection{Geometry of the Roots}
\label{subsec:geometry-of-roots}
The imaginary rescaling formula of Theorem~\ref{thm:imaginary-rescaling} gives a simple geometric description of the roots of the fundamental factors
\[
F_{d,m}(x)=\He_m^{[-d]}(1-x^d)
\]
and
\[
G_{d,m}(x)=\He_m^{[-d]}(x^d).
\]

\begin{proposition}
\label{prop:odd-root-geometry}
Let \(d\) be a positive integer and let \(m\ge1\).
If \(\alpha\) is a root of
\(
G_{d,m}(x),
\)
then
\[
\alpha^d=i\sqrt{d}\,\rho,
\]
where \(\rho\) is a real root of the probabilists' Hermite polynomial
\(\He_m(x)\).
\end{proposition}

\begin{proof}
Suppose that
\[
G_{d,m}(\alpha)=0.
\]
Then
\[
\He_m^{[-d]}(\alpha^d)=0.
\]
By Theorem~\ref{thm:imaginary-rescaling},
\[
\He_m\!\left(\frac{\alpha^d}{i\sqrt d}\right)=0.
\]
Since every root of \(\He_m(x)\) is real, there exists a real root
\(\rho\) of \(\He_m(x)\) such that
\[
\frac{\alpha^d}{i\sqrt d}=\rho.
\]
Hence
\[
\alpha^d=i\sqrt d\,\rho. \qedhere
\]
\end{proof}

\begin{proposition}
\label{prop:even-root-geometry}
Let \(d\) be a positive integer and let \(m\ge1\).
If \(\alpha\) is a root of
\(
F_{d,m}(x),
\)
then
\[
1-\alpha^d=i\sqrt{d}\,\rho,
\]
where \(\rho\) is a real root of the probabilists' Hermite polynomial
\(\He_m(x)\).
\end{proposition}

\begin{proof}
Suppose that
\[
F_{d,m}(\alpha)=0.
\]
Then
\[
\He_m^{[-d]}(1-\alpha^d)=0.
\]
Applying Theorem~\ref{thm:imaginary-rescaling},
\[
\He_m\!\left(
\frac{1-\alpha^d}{i\sqrt d}
\right)=0.
\]
Therefore
\[
\frac{1-\alpha^d}{i\sqrt d}
\]
is a real root \(\rho\) of \(\He_m(x)\), giving
\[
1-\alpha^d=i\sqrt d\,\rho. \qedhere
\] 
\end{proof}

Propositions \ref{prop:odd-root-geometry} and \ref{prop:even-root-geometry} show that the roots of odd Hermite factors
are determined by the equation
\[
x^d=i\sqrt d\,\rho,
\]
whereas the roots of even Hermite factors satisfy
\[
1-x^d=i\sqrt d\,\rho.
\]

\subsection{A Scaled Tangent Lemma}

The following lemma shows that, after scaling by the square root of an even integer, these tangent values can never be algebraic integers unless the tangent itself vanishes.

\begin{lemma}[Scaled Tangent Lemma]
\label{lem:scaled-tangent}
Let \(d\) be a positive even integer, let \(q\ge3\) be an odd integer,
and let \(j\) be an integer satisfying
\(
j\not\equiv0\pmod q.
\)
Then
\[
\frac1{\sqrt d}
\tan\!\left(\frac{j\pi}{q}\right)
\]
is not an algebraic integer.
\end{lemma}

\begin{proof}
Let
\[
\zeta=e^{2\pi ij/q},
\]
and let \(N>1\) denote the order of \(\zeta\). Since \(N\mid q\) and
\(q\) is odd, the integer \(N\) is also odd.

Using the identity
\[
\tan\theta
=
-i\,
\frac{e^{2i\theta}-1}{e^{2i\theta}+1},
\]
we obtain
\[
\tan\!\left(\frac{j\pi}{q}\right)
=
-i\,
\frac{\zeta-1}{\zeta+1}.
\]

Set
\[
\tau
=
\tan\!\left(\frac{j\pi}{q}\right).
\]
Then
\[
\tau^2
=
-
\frac{(\zeta-1)^2}{(\zeta+1)^2}.
\]

We first show that \(\tau^2\) is an algebraic integer.

Since \(\zeta\) is a root of unity, \(\zeta-1\) is an algebraic
integer. Moreover,
\[
\left|
N_{\mathbf Q(\zeta)/\mathbf Q}(1+\zeta)
\right|
=
|\Phi_N(-1)|,
\]
where \(\Phi_N(x)\) denotes the \(N\)-th cyclotomic polynomial. Since
\(N>1\) is odd,
\[
\Phi_N(-1)=1.
\]
Hence \(1+\zeta\) is an algebraic unit, and therefore
\[
(1+\zeta)^{-1}
\]
is an algebraic integer. Consequently,
\[
\tau^2
=
-(\zeta-1)^2(1+\zeta)^{-2}
\]
is an algebraic integer.

Next we compute its norm. By multiplicativity of the norm,
\[
N_{\mathbf Q(\zeta)/\mathbf Q}(\tau^2)
=
N_{\mathbf Q(\zeta)/\mathbf Q}(-1)\,
\frac{
N_{\mathbf Q(\zeta)/\mathbf Q}(1-\zeta)^2
}{
N_{\mathbf Q(\zeta)/\mathbf Q}(1+\zeta)^2
}.
\]

Since
\[
\left|
N_{\mathbf Q(\zeta)/\mathbf Q}(1-\zeta)
\right|
=
\Phi_N(1),
\]
and
\[
\Phi_N(1)
=
\begin{cases}
p,&\text{if }N=p^a\text{ is a prime power},\\
1,&\text{otherwise},
\end{cases}
\]
it follows that
\[
N_{\mathbf Q(\zeta)/\mathbf Q}(1-\zeta)
\]
is an odd integer. Furthermore,
\[
N_{\mathbf Q(\zeta)/\mathbf Q}(1+\zeta)=\pm1.
\]
Hence
\[
N_{\mathbf Q(\zeta)/\mathbf Q}(\tau^2)
\]
is an odd integer.

Suppose, for contradiction, that
\[
\frac{\tau}{\sqrt d}
\]
is an algebraic integer. Then
\[
\frac{\tau^2}{d}
\]
is also an algebraic integer. Therefore its norm is an integer:
\[
\frac{
N_{\mathbf Q(\zeta)/\mathbf Q}(\tau^2)
}{
d^{[\mathbf Q(\zeta):\mathbf Q]}
}
\in\mathbf Z.
\]
Since \(d\) is even whereas
\[
N_{\mathbf Q(\zeta)/\mathbf Q}(\tau^2)
\]
is odd, this is impossible.

Therefore
\[
\frac1{\sqrt d}
\tan\!\left(\frac{j\pi}{q}\right)
\]
cannot be an algebraic integer.
\end{proof}

The importance of Lemma~\ref{lem:scaled-tangent} is that the quantity \(\rho\) arising from Proposition~\ref{prop:even-root-geometry} is a root of a probabilists' Hermite polynomial and hence is an algebraic integer. Consequently, whenever such a root is expressed as a scaled tangent value, the lemma forces the tangent to vanish. This observation is the key ingredient in the proof of the cross-factor rigidity
theorem (Theorem \ref{thm:cross-factor-rigidity}).

\subsection{Cross-Factor Rigidity}

We now determine the possible common irreducible factors of an even Hermite factor and an odd Hermite factor. The geometric descriptions of their roots established in Section~\ref{subsec:geometry-of-roots}, together with Lemma \ref{lem:scaled-tangent}, help deduce that the interaction between the two families is extremely limited.

\begin{theorem}[Cross-Factor Rigidity]
\label{thm:cross-factor-rigidity}
Let \(d\) be a positive even integer, let \(q\) be a positive odd integer, and let \(m,k\ge1\). Define
\[
F_{d,m}(x)=\He^{[-d]}_m(1-x^d)
\]
and
\[
G_{q,k}(x)=\He^{[-q]}_k(x^q).
\]

Then the following hold.

\begin{enumerate}
\item
If \(q\ge3\), then
\[
\gcd_{\mathbb R[x]}(F_{d,m},G_{q,k})=1.
\]

\item
If \(q=1\), then
\[
\gcd_{\mathbb R[x]}(F_{d,m},G_{1,k})
\in
\{1,x^2+1\}.
\]
\end{enumerate}
\end{theorem}

\begin{proof}
Suppose that
\(
z\in\mathbb C
\)
is a common root of
\(
F_{d,m}(x)
\)
and
\(
G_{q,k}(x).
\)
By Propositions~\ref{prop:even-root-geometry} and \ref{prop:odd-root-geometry}, there exist real roots
\[
\rho\in\mathbb R,\qquad
\sigma\in\mathbb R
\]
of the ordinary Hermite polynomials
\(\He_m(x)\) and \(\He_k(x)\), respectively, such that
\[
1-z^d=i\sqrt d\,\rho,
\tag{4.1} \label{eq:4.1}
\]
and
\[
z^q=i\sqrt q\,\sigma.
\tag{4.2} \label{eq:4.2}
\]

Equation (\ref{eq:4.1}) may be rewritten as
\[
z^d=1-i\sqrt d\,\rho.
\]
Raising this equation to the \(q\)-th power and equation \eqref{eq:4.2} to the \(d\)-th power gives
\[
(1-i\sqrt d\,\rho)^q
=
(i\sqrt q\,\sigma)^d.
\tag{$\bigstar$} \label{eq:bigstar}
\]

Since \(d\) is even, the right-hand side is real. Hence
\[
(1-i\sqrt d\,\rho)^q
\in\mathbb R.
\]

Write
\[
1-i\sqrt d\,\rho
=
Re^{-i\theta},
\]
where
\[
R=\sqrt{1+d\rho^2},
\qquad
-\frac{\pi}{2}<\theta<\frac{\pi}{2}.
\]

Since the \(q\)-th power is real,
\[
q\theta\in\pi\mathbb Z.
\]

Therefore, for some integer \(j\)
\[
\theta=\frac{j\pi}{q}
\]
Since
\[
\tan\theta=\sqrt d\,\rho,
\]
we obtain
\[
\rho
=
\frac1{\sqrt d}
\tan\!\left(\frac{j\pi}{q}\right).
\tag{4.3} \label{eq:4.3}
\]

\medskip

\noindent
\textbf{Case 1.}
\(q\ge3\).

Suppose first that
\[
j\not\equiv0\pmod q.
\]

By Lemma~\ref{lem:scaled-tangent},
the quantity on the right-hand side of \eqref{eq:4.3}
cannot be an algebraic integer.

On the other hand,
\(\rho\) is a root of the monic polynomial
\(\He_m(x)\in\mathbb Z[x]\),
and is therefore an algebraic integer. This contradiction shows that
\[
j\equiv0\pmod q.
\]

Since
\[
-\frac{\pi}{2}<\theta<\frac{\pi}{2},
\]
it follows that
\[
j=0.
\]

Hence
\[
\rho=0.
\]

Equation \eqref{eq:4.1} now gives
\[
z^d=1,
\]
so
\[
|z|=1.
\]

Taking absolute values in \eqref{eq:4.2},
\[
1
=
|z|^q
=
\sqrt q\,|\sigma|,
\]
and therefore
\[
\sigma=\pm\frac1{\sqrt q}.
\]

Since \(q>1\),
\[
\sigma^2=\frac1q
\]
is a nonintegral rational number.

If \(\sigma\) were an algebraic integer, then so would be \(\sigma^2\). But the only rational algebraic integers are the ordinary integers, a contradiction. Hence no common root exists, proving
\[
\gcd_{\mathbb R[x]}(F_{d,m},G_{q,k})=1.
\]

\medskip

\noindent
\textbf{Case 2.}
\(q=1\).
Equation \eqref{eq:4.2} becomes
\[
z=i\sigma.
\]

The identity \eqref{eq:bigstar} becomes
\[
1-i\sqrt d\,\rho
=
(i\sigma)^d.
\]

Since \(d\) is even,
the right-hand side is real.

Comparing imaginary parts gives
\[
\rho=0.
\]

Hence
\[
z^d=1.
\]

Together with
\[
z=i\sigma,
\]
this implies
\[
z=\pm i.
\]

Therefore every common complex root belongs to
\[
\{\,i,-i\,\}.
\]

Since both polynomials have real coefficients, every nonconstant common divisor over
\(\mathbb R\) must have all of its roots contained in this set.

The only irreducible polynomial in
\(\mathbb R[x]\)
having these roots is
\[
x^2+1.
\]

Finally, the roots of the ordinary Hermite polynomials are simple. Hence the corresponding roots of the generalized Hermite polynomials are also simple. Since
\[
\frac{d}{dx}(1-x^d)
=
-dx^{d-1}
\]
is nonzero at
\(\pm i\),
the substitutions preserve multiplicity.

Therefore
\[
x^2+1
\]
can occur with multiplicity at most one, and hence
\[
\gcd_{\mathbb R[x]}(F_{d,m},G_{1,k})
\in
\{1,x^2+1\}.\qedhere
\]
\end{proof}

\begin{remark}
Theorem~\ref{thm:cross-factor-rigidity} is the key structural result of this section. It shows that the only irreducible polynomial that can occur simultaneously in an even Hermite factor and an odd Hermite factor
is the exceptional polynomial \(x^2+1\). This observation is the starting point for the classification of polynomial collisions in the next two sections.
\end{remark}

\subsection{Normalized Reciprocals of Odd Hermite Factors}

The proof of the quotient reduction theorem (Theorem \ref{thm:quotient-reduction}) will use normalized reciprocal polynomials to compare the lowest-degree terms of odd Hermite factors.

\begin{definition}
Let \(F(x)\in\mathbb R[x]\) be a nonzero polynomial of degree \(N\) with leading coefficient \(c_F\). The \emph{normalized reciprocal} of
\(F\) is
\[
F^{*}(t)
=
\frac{t^N F(t^{-1})}{c_F}.
\]
By construction, \(F^{*}(t)\) is a polynomial with constant term equal to \(1\).
\end{definition}

\begin{lemma}
\label{lem:normalized-reciprocal-multiplicative}
Let \(F(x),G(x)\in\mathbb R[x]\) be nonzero polynomials. Then
\[
(FG)^{*}(t)
=
F^{*}(t)\,G^{*}(t).
\]
\end{lemma}

\begin{proof}
Let
\(
\deg(F)=M, \;
\deg(G)=N,
\)
and let \(c_F,c_G\) denote their leading coefficients.
Then
\[
\deg(FG)=M+N
\]
and the leading coefficient of \(FG\) is \(c_Fc_G\). Hence
\[
(FG)^{*}(t)
=
\frac{t^{M+N}F(t^{-1})G(t^{-1})}{c_Fc_G}
=
\left(
\frac{t^M F(t^{-1})}{c_F}
\right)
\left(
\frac{t^N G(t^{-1})}{c_G}
\right),
\]
which is precisely
\[
F^{*}(t)G^{*}(t). \qedhere
\]
\end{proof}

For every positive odd integer \(d\) and every integer \(m\ge0\), recall that
\[
G_{d,m}(x)
=
\He_m^{[-d]}(x^d).
\]
Since \(G_{d,m}(x)\) is monic of degree \(md\), we define its
normalized reciprocal by
\[
S_{d,m}(t)
=
t^{md}G_{d,m}(t^{-1})
=
t^{md}\He_m^{[-d]}(t^{-d}).
\]

The following expansion will be a key ingredient.

\begin{lemma}
\label{lem:odd-reciprocal-expansion}
For every positive odd integer \(d\) and every integer \(m\ge0\),
\[
S_{d,m}(t)
=
1
+
\binom{m}{2}d\,t^{2d}
+
O(t^{4d}).
\]
In particular,
\[
S_{d,0}(t)=S_{d,1}(t)=1.
\]
\end{lemma}

\begin{proof}
Using the explicit formula for the generalized Hermite polynomials,
\[
\He_m^{[-d]}(y)
=
\sum_{j=0}^{\lfloor m/2\rfloor}
\binom{m}{2j}(2j-1)!!\,d^j
y^{m-2j},
\]
we obtain
\[
S_{d,m}(t)
=
t^{md}\He_m^{[-d]}(t^{-d})
=
\sum_{j=0}^{\lfloor m/2\rfloor}
\binom{m}{2j}(2j-1)!!\,d^j
t^{2jd}.
\]

The constant term corresponds to \(j=0\) and equals \(1\).
The first nonconstant term occurs for \(j=1\), namely
\[
\binom{m}{2}d\,t^{2d}.
\]
Every remaining summand has degree at least \(4d\). Therefore
\[
S_{d,m}(t)
=
1
+
\binom{m}{2}d\,t^{2d}
+
O(t^{4d}).
\]

Finally,
\[
\binom02=\binom12=0,
\]
so
\[
S_{d,0}(t)=S_{d,1}(t)=1.
\]
\end{proof}

\subsection{Quotient Reduction}

Theorem \ref{thm:cross-factor-rigidity} allows the equality of two partition polynomials to be reduced to two fundamentally different situations.

\begin{theorem}[Quotient Reduction]
\label{thm:quotient-reduction}
Suppose
\(
P_\lambda(x)=P_\mu(x).
\)
Then exactly one of the following holds.

\begin{enumerate}[label=\textnormal{(\roman*)}]
\item
\(
E_\lambda(x)=E_\mu(x)
\)
and
\(
O_\lambda(x)=O_\mu(x).
\)

\item
After possibly interchanging \(\lambda\) and \(\mu\), there exists an integer \(r\ge1\) such that
\[
E_\lambda(x)
=
(x^2+1)^rE_\mu(x)
\]
and
\[
O_\mu(x)
=
(x^2+1)^rO_\lambda(x).
\]
\end{enumerate}
\end{theorem}

\begin{proof}
We have
\[
P_\lambda(x)=E_\lambda(x)O_\lambda(x),
\qquad
P_\mu(x)=E_\mu(x)O_\mu(x).
\]

Let
\[
D(x)=\gcd(E_\lambda,E_\mu),
\]
chosen to be monic, and write
\[
E_\lambda=D\,A,
\qquad
E_\mu=D\,B,
\]
where
\[
\gcd(A,B)=1.
\]

Cancelling the common factor \(D\) from
\[
E_\lambda O_\lambda
=
E_\mu O_\mu
\]
gives
\[
A\,O_\lambda
=
B\,O_\mu.
\]

Since
\[
\gcd(A,B)=1,
\]
Euclid's Lemma implies that
\[
A\mid O_\mu,
\qquad
B\mid O_\lambda.
\]

Let \(f(x)\) be an irreducible factor of \(A(x)\).
Then \(f\) divides an even Hermite factor occurring in
\(E_\lambda\), and also divides an odd Hermite factor occurring in
\(O_\mu\).

By Theorem~\ref{thm:cross-factor-rigidity},
the only irreducible polynomial having this property is
\[
x^2+1.
\]
Hence every irreducible factor of \(A\) is equal to
\(x^2+1\), so
\[
A(x)=a(x^2+1)^r
\]
for some nonzero constant \(a\) and some integer \(r\ge0\).

Similarly,
\[
B(x)=b(x^2+1)^s
\]
for some nonzero constant \(b\) and integer \(s\ge0\).

Since
\[
\gcd(A,B)=1,
\]
at most one of \(r\) and \(s\) is positive.

If
\[
r=s=0,
\]
then \(A\) and \(B\) are nonzero constants.
Comparing leading coefficients in
\[
A\,O_\lambda
=
B\,O_\mu
\]
shows that \(A=B\), and therefore
\[
E_\lambda=E_\mu.
\]
Substituting this into
\[
P_\lambda=P_\mu
\]
gives
\[
O_\lambda=O_\mu,
\]
yielding case \textnormal{(i)}.

Otherwise, exactly one of \(r,s\) is positive.
After interchanging \(\lambda\) and \(\mu\), if necessary,
we may assume
\[
r\ge1,
\qquad
s=0.
\]

Again comparing leading coefficients in
\[
A\,O_\lambda
=
B\,O_\mu
\]
shows that
\[
a=b.
\]
Hence
\[
O_\mu(x)
=
(x^2+1)^rO_\lambda(x),
\]
and
\[
E_\lambda(x)
=
(x^2+1)^rE_\mu(x),
\]
which is case \textnormal{(ii)}.
\end{proof}

Theorem~\ref{thm:quotient-reduction} shows that every polynomial collision falls into one of two fundamentally different cases. If the even Hermite products coincide, then the even and odd parts separate completely. Otherwise, after possibly interchanging the two partitions, the even and odd products differ by a power of the
exceptional polynomial \(x^2+1\). These two cases are analyzed in Sections \ref{sec:equal-even-weight} and \ref{sec:exceptional-case}.

%

\section{The Case \(e_\lambda=e_\mu\)}
\label{sec:equal-even-weight}

By Theorem~\ref{thm:quotient-reduction}, every equality
\[
P_\lambda(x)=P_\mu(x)
\]
falls into one of two fundamentally different cases. In this section we consider the first case, namely that the total sizes of the even subpartitions coincide:
\[
e_\lambda=e_\mu.
\]
Our first objective is to show that this numerical equality forces the even and odd partition polynomials to agree separately. We then use these equalities to recover the even and odd subpartitions, thereby
proving that \(\lambda\) and \(\mu\) are equivalent.

\subsection{Reduction to Equal Even and Odd Products}

The first result shows that the hypothesis
\(
e_\lambda=e_\mu
\)
reduces the classification problem to the separate study of the even and odd partition polynomials.

\begin{theorem}
\label{thm:equal-even-weight-reduction}
Suppose
\(
P_\lambda(x)=P_\mu(x)
\)
and
\(
e_\lambda=e_\mu.
\)
Then
\(
E_\lambda(x)=E_\mu(x)
\)
and
\(
O_\lambda(x)=O_\mu(x).
\)
\end{theorem}

\begin{proof}
By Theorem~\ref{thm:quotient-reduction}, either
\[
E_\lambda(x)=E_\mu(x)
\]
or, after possibly interchanging \(\lambda\) and \(\mu\),
\[
E_\lambda(x)
=
(x^2+1)^rE_\mu(x)
\]
for some integer \(r\ge1\).

Suppose the second alternative holds. Since
\[
\deg(E_\lambda)=e_\lambda,
\qquad
\deg(E_\mu)=e_\mu,
\]
we obtain
\[
e_\lambda
=
e_\mu+2r,
\]
contradicting the hypothesis that
\[
e_\lambda=e_\mu.
\]
Hence
\[
E_\lambda(x)=E_\mu(x).
\]

Finally,
\[
P_\lambda(x)
=
E_\lambda(x)O_\lambda(x)
=
E_\mu(x)O_\mu(x),
\]
and cancellation yields
\[
O_\lambda(x)=O_\mu(x). \qedhere
\]
\end{proof}

\subsection{Recovery of the Even Subpartition}
\label{subsec:recovery-even-subpartition}

We now show that the even polynomial determines every even-part multiplicity. By Theorem~\ref{thm:equal-even-weight-reduction}, in the setting of this section we have
\[
E_\lambda(x)=E_\mu(x).
\]
The next result shows, more generally, that this polynomial identity alone is sufficient to recover the entire even subpartition.

For every positive even integer \(d\) and every integer \(m\ge0\),
recall that
\[
F_{d,m}(x)
=
\He_m^{[-d]}(1-x^d).
\]
When \(m=0\), we have \(F_{d,0}(x)=1\). For \(m\ge1\), the polynomial
\(F_{d,m}(x)\) has degree \(md\) and leading coefficient
\((-1)^m\). Its normalized reciprocal is therefore
\[
R_{d,m}(t)
=
(-1)^m t^{md}F_{d,m}(t^{-1})
=
(-1)^m t^{md}
\He_m^{[-d]}(1-t^{-d}).
\]
We also set
\[
R_{d,0}(t)=1.
\]

The following expansion allows the multiplicities to be recovered in increasing order of the part size.

\begin{lemma}
\label{lem:even-reciprocal-expansion}
For every positive even integer \(d\) and every integer \(m\ge0\),
\[
R_{d,m}(t)
=
1-mt^d+O(t^{2d}).
\]
In particular, the coefficient of \(t^d\) in \(R_{d,m}(t)\) is
\(-m\).
\end{lemma}

\begin{proof}
The assertion is immediate when \(m=0\), so suppose \(m\ge1\). Using the explicit formula for the generalized Hermite polynomials,
\[
\He_m^{[-d]}(y)
=
\sum_{j=0}^{\lfloor m/2\rfloor}
\binom{m}{2j}(2j-1)!!\,d^j y^{m-2j},
\]
we obtain
\[
\begin{aligned}
R_{d,m}(t)
&=
(-1)^m t^{md}
\He_m^{[-d]}(1-t^{-d})\\
&=
(-1)^m
\sum_{j=0}^{\lfloor m/2\rfloor}
\binom{m}{2j}(2j-1)!!\,d^j
t^{md}(1-t^{-d})^{m-2j}.
\end{aligned}
\]
Since
\[
1-t^{-d}=t^{-d}(t^d-1),
\]
we have
\[
t^{md}(1-t^{-d})^{m-2j}
=
t^{2jd}(t^d-1)^{m-2j}.
\]
Therefore
\[
R_{d,m}(t)
=
(-1)^m
\sum_{j=0}^{\lfloor m/2\rfloor}
\binom{m}{2j}(2j-1)!!\,d^j
t^{2jd}(t^d-1)^{m-2j}.
\]

The term corresponding to \(j=0\) is
\[
(-1)^m(t^d-1)^m
=
(1-t^d)^m
=
1-mt^d+O(t^{2d}).
\]
Every summand with \(j\ge1\) contains the factor \(t^{2jd}\), and therefore begins in degree at least \(2d\). Hence none of these terms contributes to either the constant term or the coefficient of \(t^d\).

It follows that
\[
R_{d,m}(t)
=
1-mt^d+O(t^{2d}),
\]
as required.
\end{proof}

We can now recover every even-part multiplicity.

\begin{theorem}[Recovery of the Even Subpartition]
\label{thm:recovery-even-subpartition}
Let \(\lambda\) and \(\mu\) be partitions. If
\[
E_\lambda(x)=E_\mu(x),
\]
then
\[
s_d(\lambda)=s_d(\mu)
\]
for every positive even integer \(d\). Consequently, the even subpartitions of \(\lambda\) and \(\mu\) are identical.
\end{theorem}

\begin{proof}
For every positive even integer \(d\), set
\[
m_d=s_d(\lambda),
\qquad
n_d=s_d(\mu).
\]
Then
\[
E_\lambda(x)
=
\prod_{\substack{d\ge2\\ d\text{ even}}}
F_{d,m_d}(x),
\qquad
E_\mu(x)
=
\prod_{\substack{d\ge2\\ d\text{ even}}}
F_{d,n_d}(x),
\]
where both products are finite.

Taking normalized reciprocals and using
Lemma~\ref{lem:normalized-reciprocal-multiplicative}, we obtain
\[
\prod_{\substack{d\ge2\\ d\text{ even}}}
R_{d,m_d}(t)
=
\prod_{\substack{d\ge2\\ d\text{ even}}}
R_{d,n_d}(t).
\tag{5.1} \label{eq:5.1}
\]

We prove by induction over the positive even integers that
\[
m_d=n_d.
\]

For \(d=2\), Lemma~\ref{lem:even-reciprocal-expansion} gives
\[
R_{2,m_2}(t)
=
1-m_2t^2+O(t^4).
\]
For every even integer \(e>2\),
\[
R_{e,m_e}(t)
=
1+O(t^e),
\]
and since \(e\ge4\), no such factor contributes to the coefficient of
\(t^2\) in the left-hand side of \eqref{eq:5.1}. Thus the coefficient of
\(t^2\) on the left is exactly
\[
-m_2.
\]
Similarly, the coefficient of \(t^2\) on the right is
\[
-n_2.
\]
Hence
\[
m_2=n_2.
\]

Now let \(d\ge4\) be even and suppose inductively that
\[
m_e=n_e
\]
for every positive even integer \(e<d\). The corresponding reciprocal
factors are therefore identical and may be cancelled from (5.1).
We obtain
\[
\prod_{\substack{e\ge d\\ e\text{ even}}}
R_{e,m_e}(t)
=
\prod_{\substack{e\ge d\\ e\text{ even}}}
R_{e,n_e}(t).
\tag{5.2}\label{eq:5.2}
\]

By Lemma~\ref{lem:even-reciprocal-expansion},
\[
R_{d,m_d}(t)
=
1-m_dt^d+O(t^{2d}).
\]
If \(e>d\), then
\[
R_{e,m_e}(t)
=
1+O(t^e),
\]
and hence no factor with index \(e>d\) contributes to the coefficient of \(t^d\) in the left-hand side of \eqref{eq:5.2}. Consequently, that coefficient is exactly
\[
-m_d.
\]
The same argument shows that the coefficient of \(t^d\) on the right-hand side is
\[
-n_d.
\]
Equality of the two products therefore gives
\[
m_d=n_d.
\]

By induction,
\[
s_d(\lambda)=s_d(\mu)
\]
for every positive even integer \(d\). Hence the even subpartitions of
\(\lambda\) and \(\mu\) are identical.
\end{proof}

\subsection{Recovery of Higher Odd Multiplicities}
\label{subsec:recovery-higher-odd}

By Theorem~\ref{thm:equal-even-weight-reduction}, in the present setting we also have
\[
O_\lambda(x)=O_\mu(x).
\]
We now show that this equality determines every odd part whose multiplicity is at least two.

Recall from Section~4.4 that, for every positive odd integer \(d\) and every integer \(m\ge0\),
\[
G_{d,m}(x)
=
\He_m^{[-d]}(x^d)
\]
has normalized reciprocal
\[
S_{d,m}(t)
=
t^{md}G_{d,m}(t^{-1}),
\]
and that
\[
S_{d,m}(t)
=
1+\binom{m}{2}d\,t^{2d}+O(t^{4d}).
\]
In particular,
\[
S_{d,0}(t)=S_{d,1}(t)=1.
\]

\begin{theorem}
\label{thm:recovery-higher-odd}
Let \(\lambda\) and \(\mu\) be partitions satisfying
\[
O_\lambda(x)=O_\mu(x).
\]
Then, for every positive odd integer \(d\),
\[
\max\{s_d(\lambda),\,s_d(\mu)\}\ge2
\quad\Longrightarrow\quad
s_d(\lambda)=s_d(\mu).
\]
\end{theorem}

\begin{proof}
Taking normalized reciprocals and using
Lemma~\ref{lem:normalized-reciprocal-multiplicative}, we obtain
\[
\prod_{\substack{d\ge1\\ d\text{ odd}}}
S_{d,s_d(\lambda)}(t)
=
\prod_{\substack{d\ge1\\ d\text{ odd}}}
S_{d,s_d(\mu)}(t),
\tag{5.3}\label{eq:5.3}
\]
where both products are finite.

We proceed by induction over the positive odd integers.

Fix a positive odd integer \(d\), and suppose that for every positive odd integer \(e<d\), the multiplicities have already been resolved. More precisely, if either \(s_e(\lambda)\) or \(s_e(\mu)\) is at
least two, then the induction hypothesis gives
\[
s_e(\lambda)=s_e(\mu),
\]
so the corresponding complete reciprocal factors may be canceled from both sides of \eqref{eq:5.3}. If instead
\[
s_e(\lambda),s_e(\mu)\in\{0,1\},
\]
then
\[
S_{e,s_e(\lambda)}(t)
=
S_{e,s_e(\mu)}(t)
=
1,
\]
and these factors contribute nothing.

After these cancellations, we obtain
\[
\prod_{\substack{e\ge d\\ e\text{ odd}}}
S_{e,s_e(\lambda)}(t)
=
\prod_{\substack{e\ge d\\ e\text{ odd}}}
S_{e,s_e(\mu)}(t).
\tag{5.4}\label{eq:5.4}
\]

By the reciprocal expansion from Section~4.4,
\[
S_{d,s_d(\lambda)}(t)
=
1+
\binom{s_d(\lambda)}{2}
d\,t^{2d}
+
O(t^{4d}).
\]
For every odd integer \(e>d\),
\[
S_{e,s_e(\lambda)}(t)
=
1+O(t^{2e}),
\]
and since
\[
2e>2d,
\]
none of these factors contributes to the coefficient of \(t^{2d}\) in the left-hand side of \eqref{eq:5.4}. Hence that coefficient is exactly
\[
\binom{s_d(\lambda)}{2}d.
\]

Similarly, the coefficient of \(t^{2d}\) in the right-hand side is
\[
\binom{s_d(\mu)}{2}d.
\]
Therefore
\[
\binom{s_d(\lambda)}{2}
=
\binom{s_d(\mu)}{2}.
\]

The function
\[
m\longmapsto \binom{m}{2}
\]
is injective on the positive integers, with the only ambiguity
\[
\binom{0}{2}
=
\binom{1}{2}
=
0.
\]
Consequently, if
\[
\max\{s_d(\lambda),s_d(\mu)\}\ge2,
\]
then
\[
s_d(\lambda)=s_d(\mu).
\]

This completes the induction.
\end{proof}

\subsection{Recovery of Singleton Odd Parts}
\label{subsec:recovery-singleton-odd}

Theorem~\ref{thm:recovery-higher-odd} determines every odd-part multiplicity except for the ambiguity between multiplicities \(0\) and \(1\). We now show that this remaining ambiguity is precisely the one described by Operation~\(1\).

\begin{theorem}[Recovery of Singleton Odd Parts]
\label{thm:recovery-singleton-odd}
Suppose
\[
O_\lambda(x)=O_\mu(x).
\]
After cancelling the common Hermite factors corresponding to all odd part sizes that occur with multiplicity at least two, the remaining odd parts of \(\lambda\) and \(\mu\) are related by a single legal
application of Operation~\(1\).
\end{theorem}

\begin{proof}
By Theorem~\ref{thm:recovery-higher-odd}, for every positive odd integer \(d\), if either \(s_d(\lambda)\) or \(s_d(\mu)\) is at least two, then
\[
s_d(\lambda)=s_d(\mu).
\]
Hence the corresponding generalized Hermite factors
\[
\He_{s_d(\lambda)}^{[-d]}(x^d)
\qquad\text{and}\qquad
\He_{s_d(\mu)}^{[-d]}(x^d)
\]
are identical and may be cancelled from the equality
\[
O_\lambda(x)=O_\mu(x).
\]

After all such cancellations, every remaining odd part has multiplicity either \(0\) or \(1\). Define
\[
A_\lambda
=
\left\{
d\ge1:
d\text{ is odd and }s_d(\lambda)=1
\right\},
\]
and
\[
A_\mu
=
\left\{
d\ge1:
d\text{ is odd and }s_d(\mu)=1
\right\}.
\]

For every positive odd integer \(d\) occurring with multiplicity one,
\[
\He_1^{[-d]}(x^d)=x^d.
\]
Therefore the residual polynomial equality is
\begin{equation*}
\label{eq:singleton-residual-products}
\prod_{d\in A_\lambda}x^d
=
\prod_{e\in A_\mu}x^e.
\end{equation*}
Equivalently,
\[
x^{\sum_{d\in A_\lambda}d}
=
x^{\sum_{e\in A_\mu}e},
\]
and hence
\begin{equation}
\label{eq:singleton-total-sums}
\sum_{d\in A_\lambda}d
=
\sum_{e\in A_\mu}e. \tag{5.5} 
\end{equation}

Set
\[
A=A_\lambda\setminus A_\mu,
\qquad
B=A_\mu\setminus A_\lambda.
\]
Then \(A\) and \(B\) are disjoint sets of positive odd integers.
Removing from both sides of
\eqref{eq:singleton-total-sums} the contribution of the common set
\(A_\lambda\cap A_\mu\) gives
\begin{equation}
\label{eq:singleton-operation-one-sums}
\sum_{a\in A}a
=
\sum_{b\in B}b. \tag{5.6}
\end{equation}

Every part in \(A\) occurs in \(\lambda\) with multiplicity exactly
one. We claim that every part in \(B\) is absent from \(\lambda\).

Indeed, let \(b\in B\). By definition,
\[
s_b(\mu)=1
\]
and
\[
s_b(\lambda)\neq1.
\]
If
\[
s_b(\lambda)\ge2,
\]
then Theorem~\ref{thm:recovery-higher-odd} would imply
\[
s_b(\lambda)=s_b(\mu),
\]
contradicting
\[
s_b(\mu)=1.
\]
Therefore
\[
s_b(\lambda)=0.
\]

Thus the distinct odd parts belonging to \(A\) may be removed from \(\lambda\), while the distinct odd parts belonging to \(B\) may be inserted. Every removed part has multiplicity exactly one before the
move, every inserted part is absent before the move, and \eqref{eq:singleton-operation-one-sums} shows that the total sum of the removed parts equals the total sum of the inserted parts. Consequently, this replacement is a legal application of Operation~\(1\).

All odd part sizes outside \(A\cup B\) already occur with the same multiplicity in the two partitions. Hence this single application of Operation~\(1\) transforms the odd subpartition of \(\lambda\) into
the odd subpartition of \(\mu\).
\end{proof}

\subsection{Completion of the Case \(e_\lambda=e_\mu\)}
\label{subsec:completion-equal-even-weight}

We now combine the preceding results to complete the proof of the equal even-weight case.

\begin{theorem}
\label{thm:equal-even-weight-classification}
Let \(\lambda\) and \(\mu\) be partitions of the same integer \(n\).
Suppose
\[
P_\lambda(x)=P_\mu(x)
\]
and
\[
e_\lambda=e_\mu.
\]
Then
\[
\lambda\sim\mu.
\]
Moreover, \(\lambda\) and \(\mu\) are related by a finite sequence of applications of Operation~\(1\).
\end{theorem}

\begin{proof}
Suppose
\[
P_\lambda(x)=P_\mu(x)
\]
and
\[
e_\lambda=e_\mu.
\]

By Theorem~\ref{thm:equal-even-weight-reduction},
\[
E_\lambda(x)=E_\mu(x)
\]
and
\[
O_\lambda(x)=O_\mu(x).
\]

Applying
Theorem~\ref{thm:recovery-even-subpartition},
we conclude that the even subpartitions of \(\lambda\) and \(\mu\) are identical.

Next, Theorem~\ref{thm:recovery-higher-odd} shows that every odd part occurring with multiplicity at least two
appears in both partitions with exactly the same multiplicity. Cancelling these common odd Hermite factors leaves only odd parts occurring with multiplicity one.

Finally,
Theorem~\ref{thm:recovery-singleton-odd}
shows that the remaining singleton odd parts are related by a single legal application of Operation~\(1\).

Thus the even parts already agree, every repeated odd part already agrees, and the remaining singleton odd parts differ only by Operation~\(1\). Consequently, \(\lambda\) and \(\mu\) are related by
a finite sequence of applications of Operation~\(1\).

Therefore
\[
\lambda\sim\mu. \qedhere
\]
\end{proof}

\section{The Reciprocal Core and Logarithmic Linearization}
\label{sec:unequal-even-weight}

A significant obstacle in classifying collisions of twisted Foulkes character polynomials is that generalized Hermite polynomials do not admit a naive root-separation principle: distinct generalized Hermite
polynomials with different parameters may have common nonzero roots. For example,
\[
\He^{[-3]}_{2}(y)=y^{2}+3,
\qquad
\He^{[-1]}_{3}(y)=y(y^{2}+3),
\]
share the nonzero roots \(\pm i\sqrt{3}\). Consequently, the classification cannot be obtained by identifying generalized Hermite factors from their roots alone. The challenge is therefore not to distinguish generalized Hermite roots, but to prove that such local root coincidences cannot give rise to any global collisions beyond the two explicit combinatorial operations described in this paper. In this section we develop the analytic machinery needed for the exceptional case $e_\lambda \neq e_\mu$. In the next section, this machinery will be used to recover the exact multiplicities of the parts \(2\) and \(4\), and subsequently all remaining even multiplicities.

Recall
\[
F_{d,m}(x)
=
\He^{[-d]}_m(1-x^d).
\]
Thus, if \(\nu\) is a partition, then its even-part polynomial is
\[
E_\nu(x)
=
\prod_{\substack{d\geq 2\\ d\text{ even}}}
F_{d,s_d(\nu)}(x),
\]
where only finitely many factors are nonconstant.

For \(m\geq 0\), define
\[
\varepsilon(m)
=
\begin{cases}
0,&m\text{ is even},\\
1,&m\text{ is odd}.
\end{cases}
\]

\subsection{Cyclotomic Decomposition of the Even Hermite Factors}
\label{subsec:cyclotomic-decomposition}

We begin by separating the unit-circle roots of an even Hermite factor
from its remaining roots.

\begin{lemma}
\label{lem:unit-circle-roots-even-factor}
Let \(d\) be a positive even integer and let \(m\geq 0\). A complex
number \(z\) with \(\lvert z\rvert=1\) is a root of
\[
F_{d,m}(x)
=
\He^{[-d]}_m(1-x^d)
\]
if and only if \(m\) is odd and
\(
z^d=1.
\)
Moreover, every such root is simple.
\end{lemma}

\begin{proof}
The roots of \(\He^{[-d]}_m(y)\) are purely imaginary. Indeed, using the
ordinary probabilists' Hermite polynomial \(\He_m\), one has
\[
\He^{[-d]}_m(y)
=
(i\sqrt d)^m
\He_m\left(\frac{y}{i\sqrt d}\right).
\]
Hence every root of \(\He^{[-d]}_m\) has the form
\[
i\sqrt d\,\rho,
\]
where \(\rho\) is a real root of \(\He_m\).

Suppose that \(\lvert z\rvert=1\) and
\[
F_{d,m}(z)=0.
\]
Then
\[
1-z^d=i\sqrt d\,\rho
\]
for some real Hermite root \(\rho\). Therefore
\[
z^d=1-i\sqrt d\,\rho.
\]
Taking absolute values gives
\[
1
=
\lvert z^d\rvert^2
=
\left|1-i\sqrt d\,\rho\right|^2
=
1+d\rho^2.
\]
Thus \(\rho=0\), and consequently
\[
z^d=1.
\]

The ordinary Hermite polynomial \(\He_m\) has \(0\) as a root if and only if \(m\) is odd. Hence unit-circle roots occur precisely when \(m\) is odd, and in that case they are exactly the \(d\)-th roots of unity.

The root \(0\) of \(\He_m\) is simple. Moreover,
\[
\frac{d}{dx}(1-x^d)
=
-dx^{d-1},
\]
which is nonzero at every \(d\)-th root of unity. Therefore the substitution \(y=1-x^d\) preserves the multiplicity of these roots, and each unit-circle root of \(F_{d,m}\) is simple.
\end{proof}

\begin{definition}
\label{def:hermite-core}
For every positive even integer \(d\) and every \(m\geq 0\), define the \emph{Hermite core} \(C_{d,m}(x)\) by
\[
F_{d,m}(x)
=
(1-x^d)^{\varepsilon(m)}C_{d,m}(x).
\]
\end{definition}

\begin{proposition}
\label{prop:core-no-cyclotomic-factors}
The polynomial \(C_{d,m}(x)\) has no roots on the unit circle. In particular, it has no cyclotomic polynomial as a divisor.
\end{proposition}

\begin{proof}
By Lemma~\ref{lem:unit-circle-roots-even-factor}, the unit-circle roots of \(F_{d,m}\) are present only when \(m\) is odd, and in that case they are precisely the simple roots of \(1-x^d\). Dividing by \(1-x^d\)
therefore removes all unit-circle roots of \(F_{d,m}\). Hence \(C_{d,m}\) has no unit-circle roots.

Every root of a cyclotomic polynomial lies on the unit circle. Thus \(C_{d,m}\) cannot have a cyclotomic polynomial as a divisor.
\end{proof}

For a partition \(\nu\), define
\[
U_\nu(x)
=
\prod_{\substack{d\geq 2\\ d\text{ even}}}
(1-x^d)^{\varepsilon(s_d(\nu))}
\]
and
\[
C_\nu(x)
=
\prod_{\substack{d\geq 2\\ d\text{ even}}}
C_{d,s_d(\nu)}(x).
\]
Then
\[
E_\nu(x)=U_\nu(x)C_\nu(x).
\]

The polynomial \(U_\nu\) is the complete cyclotomic part of \(E_\nu\), while \(C_\nu\) has no cyclotomic divisors.

\subsection{Cyclotomic Multiset Comparison}
\label{subsec:cyclotomic-multiset-comparison}

We next record the cyclotomic consequence of an identity of the form
\[
E_\lambda(x)
=
(x^2+1)E_\mu(x).
\]

For a positive integer \(r\), let \(\Phi_r(x)\) denote the \(r\)-th cyclotomic polynomial. Recall that
\[
1-x^d
=
-\prod_{r\mid d}\Phi_r(x).
\]

For every positive integer \(r\), the multiplicity of \(\Phi_r\) in
\(U_\nu\) is
\[
c_r(\nu)
=
\sum_{\substack{d\geq 2\\ d\text{ even}\\ r\mid d}}
\varepsilon(s_d(\nu)).
\]

\begin{proposition}
\label{prop:cyclotomic-parity-exceptional}
Let \(\lambda\) and \(\mu\) be partitions satisfying
\[
E_\lambda(x)
=
(x^2+1)E_\mu(x).
\]
Then
\[
s_2(\lambda)\equiv 0\pmod 2,
\qquad
s_2(\mu)\equiv 1\pmod 2,
\]
and
\[
s_4(\lambda)\equiv 1\pmod 2,
\qquad
s_4(\mu)\equiv 0\pmod 2.
\]
Moreover,
\[
s_d(\lambda)\equiv s_d(\mu)\pmod 2
\]
for every positive even integer \(d\neq 2,4\).
\end{proposition}

\begin{proof}
Since
\[
x^2+1=\Phi_4(x),
\]
we have,
\begin{align*}
E_\lambda(x)
&=
(x^2+1)E_\mu(x).\\
\implies \qquad  U_\lambda(x)C_\lambda(x)
&=
\Phi_4(x)\,U_\mu(x)C_\mu(x).
\end{align*}
By Proposition~\ref{prop:core-no-cyclotomic-factors}, the Hermite cores \(C_\lambda(x)\) and \(C_\mu(x)\) contain no cyclotomic factors. Consequently, the multiplicity of each cyclotomic polynomial \(\Phi_r(x)\) is determined entirely by the cyclotomic parts \(U_\lambda(x)\) and \(U_\mu(x)\).

For \(r\ge1\), let \(c_r(\nu)\) denote the multiplicity of \(\Phi_r(x)\) in \(U_\nu(x)\). Comparing the multiplicity of \(\Phi_r(x)\) on the two sides of the above identity, we obtain
\[
c_r(\lambda)-c_r(\mu)
=
\begin{cases}
1,&r=4,\\
0,&r\neq4,
\end{cases}
\]
since the only additional cyclotomic factor on the right-hand side is
the single factor \(\Phi_4(x)\).

Set
\[
\delta_d
=
\varepsilon(s_d(\lambda))
-
\varepsilon(s_d(\mu)).
\]
Then
\[
\sum_{\substack{d\geq 2\\ d\text{ even}\\ r\mid d}}
\delta_d
=
\begin{cases}
1,&r=4,\\
0,&r\neq 4.
\end{cases}
\tag{\(\ast\)} \label{eq:star}
\]

Note that only finitely many \(\delta_d\) are nonzero. Let \(D\) be the largest
even integer such that \(\delta_D\neq 0\). Substituting \(r=D\) into \eqref{eq:star}, the only nonzero term in the sum is \(\delta_D\). Hence
\[
\delta_D
=
\begin{cases}
1,&D=4,\\
0,&D\neq 4.
\end{cases}
\]
Therefore \(D=4\), so
\[
\delta_d=0
\qquad
\text{for every even }d>4.
\]

Substituting \(r=4\) into \eqref{eq:star} now gives
\[
\delta_4=1.
\]
Substituting \(r=2\) gives
\[
\delta_2+\delta_4=0,
\]
and hence
\[
\delta_2=-1.
\]

Therefore
\[
\varepsilon(s_2(\lambda))=0,
\qquad
\varepsilon(s_2(\mu))=1,
\]
and
\[
\varepsilon(s_4(\lambda))=1,
\qquad
\varepsilon(s_4(\mu))=0.
\]
For every other positive even \(d\),
\[
\varepsilon(s_d(\lambda))
=
\varepsilon(s_d(\mu)).
\]
This is equivalent to the stated parity relations.
\end{proof}

\begin{corollary}
\label{cor:core-equality-exceptional}
Under the hypotheses of
Proposition~\ref{prop:cyclotomic-parity-exceptional}, we have
\[
C_\lambda(x)=C_\mu(x).
\]
\end{corollary}

\begin{proof}
Recall
\[
E_\lambda=U_\lambda C_\lambda
\qquad\text{and}\qquad
E_\mu=U_\mu C_\mu.
\]
Since both $U_\lambda(x)$ and $(x^2+1)U_\mu(x)$ have constant
term $1$, the cyclotomic comparison gives
\[
U_\lambda(x)=(x^2+1)U_\mu(x).
\]
Substituting this into
\[
E_\lambda=(x^2+1)E_\mu
\]
and canceling the nonzero polynomial
\[
(x^2+1)U_\mu
\]
gives
\[
C_\lambda=C_\mu. \qedhere
\]
\end{proof}

\subsection{Normalized Reciprocals of the Hermite Cores}
\label{subsec:normalized-reciprocal-cores}

Define the normalized reciprocal of \(F_{d,m}\) by
\[
R_{d,m}(t)
=
(-1)^m t^{dm}F_{d,m}(t^{-1}).
\]

Using
\[
F_{d,m}(x)
=
\He^{[-d]}_m(1-x^d),
\]
we may write
\[
R_{d,m}(t)
=
(-1)^m t^{dm}
\He^{[-d]}_m(1-t^{-d}).
\]

Since the normalized reciprocal of \(1-x^d\) is \(1-t^d\), define
\[
T_{d,m}(t)
=
\frac{R_{d,m}(t)}
     {(1-t^d)^{\varepsilon(m)}}.
\]
Thus \(T_{d,m}\) is the normalized reciprocal of the Hermite core
\(C_{d,m}\).

\begin{lemma}
\label{lem:even-reciprocal-first-terms}
For every positive even integer \(d\) and every \(m\geq 0\),
\[
R_{d,m}(t)
=
1-mt^d
+
(d+1)\binom{m}{2}t^{2d}
+
O(t^{3d}).
\]
\end{lemma}

\begin{proof}
Using the explicit formula
\[
\He^{[-d]}_m(y)
=
\sum_{j=0}^{\lfloor m/2\rfloor}
\binom{m}{2j}(2j-1)!!\,d^j y^{m-2j},
\]
we substitute \(y=1-t^{-d}\). The highest-degree contribution in
\(t^{-d}\) comes from \(j=0\):
\[
(1-t^{-d})^m
=
(-1)^m t^{-dm}
\left(
1-mt^d+\binom{m}{2}t^{2d}+O(t^{3d})
\right).
\]

The term \(j=1\) contributes
\[
\binom{m}{2}d(1-t^{-d})^{m-2}.
\]
After multiplication by \((-1)^m t^{dm}\), its contribution begins in
degree \(2d\), with coefficient
\[
d\binom{m}{2}.
\]

Every term with \(j\geq 2\) begins in degree at least \(4d\).
Consequently,
\[
R_{d,m}(t)
=
1-mt^d
+
\left(
\binom{m}{2}
+
d\binom{m}{2}
\right)t^{2d}
+
O(t^{3d}),
\]
which gives
\[
R_{d,m}(t)
=
1-mt^d
+
(d+1)\binom{m}{2}t^{2d}
+
O(t^{3d}).\qedhere
\]
\end{proof}

\begin{lemma}
\label{lem:core-reciprocal-first-term}
For every positive even integer \(d\) and every \(m\geq 0\),
\[
T_{d,m}(t)
=
1-
\bigl(m-\varepsilon(m)\bigr)t^d
+
O(t^{2d}).
\]
Equivalently,
\[
T_{d,m}(t)
=
1-
2\left\lfloor\frac m2\right\rfloor t^d
+
O(t^{2d}).
\]
\end{lemma}

\begin{proof}
If \(m\) is even, then \(\varepsilon(m)=0\), so
\[
T_{d,m}(t)=R_{d,m}(t),
\]
and the result follows from
Lemma~\ref{lem:even-reciprocal-first-terms}.

If \(m\) is odd, then
\[
T_{d,m}(t)
=
\frac{R_{d,m}(t)}{1-t^d}.
\]
Since
\[
\frac{1}{1-t^d}
=
1+t^d+O(t^{2d}),
\]
we obtain
\[
\begin{aligned}
T_{d,m}(t)
&=
\left(
1-mt^d+O(t^{2d})
\right)
\left(
1+t^d+O(t^{2d})
\right)\\
&=
1-(m-1)t^d+O(t^{2d}).
\end{aligned}
\]
Since \(\varepsilon(m)=1\), this is
\[
1-\bigl(m-\varepsilon(m)\bigr)t^d+O(t^{2d}). \qedhere
\]
\end{proof}

\begin{proposition}
\label{prop:core-reciprocal-product}
Let \(\lambda\) and \(\mu\) be partitions satisfying
\(
C_\lambda(x)=C_\mu(x).
\)
Then
\[
\prod_{\substack{d\geq2\\d\text{ even}}}
T_{d,s_d(\lambda)}(t)
=
\prod_{\substack{d\geq2\\d\text{ even}}}
T_{d,s_d(\mu)}(t).
\]
\end{proposition}

\begin{proof}
Normalized reciprocation is multiplicative for monic polynomials.
Therefore, taking normalized reciprocals of
\[
C_\lambda(x)=C_\mu(x)
\]
gives the stated identity.
\end{proof}

\subsection{Logarithmic Linearization}
\label{subsec:logarithmic-linearization}

Every polynomial \(T_{d,m}(t)\) has constant term \(1\). Therefore its
formal logarithm is well defined in \(\mathbb Q[[t]]\).

\begin{definition}
If
\(
A(t)\in 1+t\mathbb Q[[t]],
\)
define
\[
\log A(t)
=
\sum_{r\geq1}
\frac{(-1)^{r+1}}{r}
\bigl(A(t)-1\bigr)^r.
\]
\end{definition}

\begin{lemma}
\label{lem:formal-log-product}
Let
\(
A_1(t),\ldots,A_r(t),
B_1(t),\ldots,B_s(t)
\in 1+t\mathbb Q[[t]].
\)
If
\[
\prod_{i=1}^r A_i(t)
=
\prod_{j=1}^s B_j(t),
\]
then
\[
\sum_{i=1}^r\log A_i(t)
=
\sum_{j=1}^s\log B_j(t).
\]
\end{lemma}

\begin{proof}
The formal logarithm satisfies
\[
\log(AB)=\log A+\log B
\]
for power series \(A,B\in 1+t\mathbb Q[[t]]\). Applying this identity
repeatedly to the two products proves the result.
\end{proof}

\begin{lemma}
\label{lem:logarithmic-divisor-support}
For every positive even integer \(d\) and every \(m\geq0\),
\[
\log T_{d,m}(t)
\in
t^d\mathbb Q[[t^d]].
\]
Consequently, in an identity
\[
\sum_{\substack{d\geq2\\d\text{ even}}}
\log T_{d,a_d}(t)
=
\sum_{\substack{d\geq2\\d\text{ even}}}
\log T_{d,b_d}(t),
\]
the coefficient of \(t^N\) receives contributions only from those even
integers \(d\) satisfying
\[
d\mid N.
\]
\end{lemma}

\begin{proof}
By construction, \(R_{d,m}(t)\) is a polynomial in \(t^d\), and so is
\[
(1-t^d)^{-\varepsilon(m)}.
\]
Hence
\[
T_{d,m}(t)\in 1+t^d\mathbb Q[t^d].
\]
Its formal logarithm therefore belongs to
\[
t^d\mathbb Q[[t^d]].
\]

A term of degree \(N\) can occur in
\(\log T_{d,m}(t)\) only if \(N\) is a positive multiple of \(d\).
Therefore only part sizes \(d\mid N\) contribute to the coefficient of
\(t^N\).
\end{proof}

\begin{corollary}
\label{cor:logarithmic-core-identity}
If
\(
C_\lambda(x)=C_\mu(x),
\)
then
\[
\sum_{\substack{d\geq2\\d\text{ even}}}
\log T_{d,s_d(\lambda)}(t)
=
\sum_{\substack{d\geq2\\d\text{ even}}}
\log T_{d,s_d(\mu)}(t).
\]
Moreover, the coefficient of \(t^N\) in this identity involves only the
even part sizes dividing \(N\).
\end{corollary}

\subsection{Root Factorization of the Reciprocal Cores}
\label{subsec:root-factorization}

We now express the reciprocal core polynomials in terms of the ordinary
Hermite roots.

\begin{proposition}
\label{prop:reciprocal-root-factorization}
Let \(\rho\) range over the roots of the ordinary Hermite polynomial
\(\He_m\). Set
\[
u=t^d.
\]
Then
\[
R_{d,m}(u)
=
\prod_{\He_m(\rho)=0}
\left(
1-\left(1-i\sqrt d\,\rho\right)u
\right).
\]

After removing the factor corresponding to \(\rho=0\) when \(m\) is
odd, we have
\[
T_{d,m}(u)
=
\prod_{\substack{\He_m(\rho)=0\\ \rho\neq0}}
\left(
1-\left(1-i\sqrt d\,\rho\right)u
\right).
\]
\end{proposition}

\begin{proof}
By Theorem~\ref{thm:imaginary-rescaling}, the roots of
\(\He_m^{[-d]}(y)\) are
\(
i\sqrt d\,\rho,
\)
where \(\rho\) ranges over the roots of \(\He_m\). Since
\(\He_m^{[-d]}\) is monic,
\[
\He_m^{[-d]}(y)
=
\prod_{\He_m(\rho)=0}
\left(y-i\sqrt d\,\rho\right).
\]
Therefore
\[
F_{d,m}(x)
=
\prod_{\He_m(\rho)=0}
\left(
1-x^d-i\sqrt d\,\rho
\right).
\]

Substituting \(x=t^{-1}\) and multiplying by
\((-1)^mt^{dm}\), we obtain
\[
R_{d,m}(t)
=
\prod_{\He_m(\rho)=0}
\left(
1-\left(1-i\sqrt d\,\rho\right)t^d
\right).
\]
Writing \(u=t^d\) gives the first identity.

If \(m\) is odd, the root \(\rho=0\) contributes the factor
\[
1-u.
\]
Dividing by \(1-u\) removes precisely this factor. If \(m\) is even, there is no zero root. This proves the factorization of \(T_{d,m}\).
\end{proof}

For \(q\geq0\), define
\[
A_{d,q}(u)
=
T_{d,2q+1}(u)
\]
and
\[
B_{d,q}(u)
=
T_{d,2q}(u).
\]
Both polynomials have degree \(2q\) in \(u\).

Let
\[
0<\rho_{m,1}<\cdots<\rho_{m,\lfloor m/2\rfloor}
\]
denote the positive roots of \(\He_m\). Then
\[
A_{d,q}(u)
=
\prod_{j=1}^{q}
(1-\alpha_j u)(1-\overline{\alpha_j}u),
\]
where
\[
\alpha_j
=
1-i\sqrt d\,\rho_{2q+1,j},
\]
and
\[
B_{d,q}(u)
=
\prod_{j=1}^{q}
(1-\beta_j u)(1-\overline{\beta_j}u),
\]
where
\[
\beta_j
=
1-i\sqrt d\,\rho_{2q,j}.
\]


\begin{corollary}
\label{cor:logarithmic-power-sums}
Write
\[
\log\frac{A_{d,q}(u)}{B_{d,q}(u)}
=
\sum_{n\geq1}c_n(d,q)u^n.
\]
Then
\[
c_n(d,q)
=
-\frac{S_n(d,q)}{n},
\]
where
\[
S_n(d,q)
=
\sum_{j=1}^{q}
\left(
\alpha_j^n+\overline{\alpha_j}^{\,n}
\right)
-
\sum_{j=1}^{q}
\left(
\beta_j^n+\overline{\beta_j}^{\,n}
\right).
\]
In particular, \(S_n(d,q)\) is a linear recurrence sequence over a
field of characteristic zero.
\end{corollary}

\begin{proof}
Since \(A_{d,q}(0)=B_{d,q}(0)=1\), the formal logarithm of
\(A_{d,q}(u)/B_{d,q}(u)\) is well defined. By the root factorizations,
\[
A_{d,q}(u)
=
\prod_{j=1}^{q}
(1-\alpha_j u)(1-\overline{\alpha_j}u),
\]
and
\[
B_{d,q}(u)
=
\prod_{j=1}^{q}
(1-\beta_j u)(1-\overline{\beta_j}u).
\]
For any complex number \(\gamma\),
\[
\log(1-\gamma u)
=
-\sum_{n\geq1}\frac{\gamma^n}{n}u^n.
\]
Therefore
\[
\log\frac{A_{d,q}(u)}{B_{d,q}(u)}
=
-\sum_{n\geq1}
\frac{1}{n}
\left[
\sum_{j=1}^{q}
\left(
\alpha_j^n+\overline{\alpha_j}^{\,n}
\right)
-
\sum_{j=1}^{q}
\left(
\beta_j^n+\overline{\beta_j}^{\,n}
\right)
\right]u^n.
\]
Comparing coefficients of \(u^n\) gives
\[
c_n(d,q)=-\frac{S_n(d,q)}{n}.
\]

Finally, \(S_n(d,q)\) is a finite linear combination of exponential sequences of the form \(\gamma^n\). Hence it satisfies a constant-coefficient linear recurrence over the field generated by the numbers
\[
\alpha_j,\overline{\alpha_j},\beta_j,\overline{\beta_j},
\qquad 1\leq j\leq q,
\]
which has characteristic zero.
\end{proof}


\subsection{External Results on Linear Recurrences and Primes}
\label{subsec:external-results}

We record two classical theorems that will be used in the proof of the prime-coefficient lemma in the next section.

\begin{theorem}[Skolem--Mahler--Lech]\cite[p.25]{Skolem-Mahler-Lech}
\label{thm:skolem-mahler-lech}
Let \((u_n)_{n\geq0}\) be a linear recurrence sequence over a field of
characteristic zero. Then the zero set
\[
\{n\geq0:u_n=0\}
\]
is the union of a finite set and finitely many arithmetic progressions.
\end{theorem}

\begin{remark}
After choosing a common modulus \(M\) for the finitely many arithmetic progressions in Theorem~\ref{thm:skolem-mahler-lech}, membership in the zero set is, outside a finite exceptional set, determined by the residue
class modulo \(M\).
\end{remark}

\begin{theorem}[Dirichlet's theorem on primes in arithmetic progressions]\cite[Theorem 7.9]{Apostol-NT}
\label{thm:dirichlet-primes}
Let \(a\) and \(M\) be positive integers satisfying
\(
\gcd(a,M)=1.
\)
Then the arithmetic progression
\[
a,\ a+M,\ a+2M,\ldots
\]
contains infinitely many prime numbers.
\end{theorem}

The probabilists' Hermite polynomials $\{\He_n(x)\}_{n\geq 0}$ form an orthogonal polynomial sequence. We shall also use the standard strict interlacing property of the roots of consecutive Hermite polynomials.

\begin{theorem}[Strict interlacing of Hermite roots]\cite[Theorem 3.3.2]{Orthogonal-polynomials-Gabor}
\label{thm:hermite-interlacing}
The roots of every ordinary Hermite polynomial \(\He_m\) are real and simple. Moreover, the roots of \(\He_m\) and \(\He_{m+1}\) strictly interlace.

In particular, if \(\rho_m^{\max}\) denotes the largest positive root
of \(\He_m\), then
\[
\rho_{m+1}^{\max}
>
\rho_m^{\max}, \qquad m\geq 2.
\]
\end{theorem}

\section{The Case \(e_\lambda\neq e_\mu\)}
\label{sec:exceptional-case}

In this section, we complete the necessity proof in the case where the
even weights of the two partitions are unequal. Let \(\lambda,\mu\vdash n\), and suppose that
\(
P_\lambda(x)=P_\mu(x)
\)
and
\(
e_\lambda\neq e_\mu.
\)
After interchanging \(\lambda\) and \(\mu\), if necessary, we may assume
that
\(
e_\lambda>e_\mu.
\)

Our goal is to prove that
\[
\lambda\sim\mu.
\]

The proof proceeds in several stages. We first establish the prime-coefficient nonvanishing lemma. We then apply it to show that the exceptional even-part discrepancy consists of a unique part
\(4\) on the \(\lambda\)-side and a unique part \(2\) on the \(\mu\)-side. The odd-part discrepancy is supplied by two parts of size \(1\), up to Operation~\(1\). The even part transformation is precisely
Operation~\(2\).

\subsection{The Prime-Coefficient Nonvanishing Lemma}
\label{subsec:prime-coefficient-lemma}

Let \(d\) be a positive even integer and let \(q\geq 0\). Recall the
polynomials
\[
A_{d,q}(u)
=
T_{d,2q+1}(u)
\]
and
\[
B_{d,q}(u)
=
T_{d,2q}(u),
\]
where \(u=t^d\).

Write
\[
\log\frac{A_{d,q}(u)}{B_{d,q}(u)}
=
\sum_{n\geq1}c_n(d,q)u^n.
\]

We now prove the key nonvanishing result.

\begin{theorem}[Prime-coefficient nonvanishing]
\label{thm:prime-coefficient-nonvanishing}
Let \(d\) be a positive even integer and let \(q>0\). Then
\[
c_p(d,q)\neq0
\]
for infinitely many odd primes \(p\).
\end{theorem}

\begin{proof}
By Corollary~\ref{cor:logarithmic-power-sums},
\[
c_n(d,q)
=
-\frac{S_n(d,q)}{n},
\]
where
\[
S_n(d,q)
=
\sum_{j=1}^{q}
\left(
\alpha_j^n+\overline{\alpha_j}^{\,n}
\right)
-
\sum_{j=1}^{q}
\left(
\beta_j^n+\overline{\beta_j}^{\,n}
\right),
\]
with
\[
\alpha_j
=
1-i\sqrt d\,\rho_{2q+1,j}
\]
and
\[
\beta_j
=
1-i\sqrt d\,\rho_{2q,j}.
\]
Here
\[
0<\rho_{m,1}<\cdots<\rho_{m,\lfloor m/2\rfloor}
\]
are the positive roots of the ordinary Hermite polynomial \(\He_m\).

Suppose, for a contradiction, that
\[
c_p(d,q)=0
\]
for all sufficiently large odd primes \(p\). Equivalently,
\[
S_p(d,q)=0
\]
for all sufficiently large odd primes \(p\).

The sequence
\[
\bigl(S_n(d,q)\bigr)_{n\geq1}
\]
is a linear recurrence sequence over a field of characteristic zero. By the Skolem--Mahler--Lech theorem,
Theorem~\ref{thm:skolem-mahler-lech}, its zero set is the union of a finite set and finitely many arithmetic progressions. Let \(M\) be a common multiple of the moduli of these arithmetic progressions. Then, outside a finite exceptional set, membership in the zero set is determined entirely by the residue class modulo \(M\). By Dirichlet's theorem, Theorem~\ref{thm:dirichlet-primes}, there are infinitely many primes
congruent to \(1\pmod M\). By assumption, every sufficiently large such prime is a zero of \(S_n(d,q)\). Hence the residue class \(1\pmod M\) must be one of the eventual zero classes. Therefore
\[
S_{1+kM}(d,q)=0
\]
for every sufficiently large integer \(k\).
The subsequence \(S_{1+kM}(d,q)\) is an exponential polynomial in
\(k\). Explicitly,
\[
S_{1+kM}(d,q)
=
\sum_{j=1}^{q}
\alpha_j\left(\alpha_j^M\right)^k
+
\sum_{j=1}^{q}
\overline{\alpha_j}
\left(\overline{\alpha_j}^{\,M}\right)^k
-
\sum_{j=1}^{q}
\beta_j\left(\beta_j^M\right)^k
-
\sum_{j=1}^{q}
\overline{\beta_j}
\left(\overline{\beta_j}^{\,M}\right)^k.
\tag{7.1}\label{eq:7.1}
\]
All of the exponential bases occurring here are nonzero. After grouping together equal bases, an exponential polynomial with distinct nonzero bases that vanishes for all sufficiently large integers must vanish
identically, by the linear independence of distinct exponential sequences. Hence
\[
S_{1+kM}(d,q)=0
\]
for every \(k\geq0\).
Let
\[
\rho_{\max}
=
\rho_{2q+1,q}
\]
be the largest positive root of \(\He_{2q+1}\), and set
\[
\alpha
=
1-i\sqrt d\,\rho_{\max}.
\]
By the strict interlacing of consecutive Hermite roots,
Theorem~\ref{thm:hermite-interlacing},
\[
\rho_{2q+1,q}
>
\rho_{2q,q}.
\]
Since
\[
|1-i\sqrt d\,\rho|
=
\sqrt{1+d\rho^2},
\]
which is strictly increasing for \(\rho>0\), it follows that
\[
|\alpha|
=
\sqrt{1+d\rho_{\max}^2}
\]
is strictly larger than the modulus of every \(\beta_j\) and every \(\alpha_j\) with \(j<q\). Thus the only terms of maximal modulus in \eqref{eq:7.1} are those with bases
\[
\alpha^M
\qquad\text{and}\qquad
\overline{\alpha}^{\,M}.
\]
Suppose first that
\[
\alpha^M
\neq
\overline{\alpha}^{\,M}.
\]
Then these are distinct exponential bases. Their coefficients in \eqref{eq:7.1}
are respectively
\[
\alpha
\qquad\text{and}\qquad
\overline{\alpha},
\]
both of which are nonzero. Hence their contributions cannot disappear from an identically zero exponential polynomial. This is a contradiction.

Suppose instead that
\[
\alpha^M
=
\overline{\alpha}^{\,M}.
\]
Then the two maximal-modulus terms combine to give
\[
\left(\alpha+\overline{\alpha}\right)
\left(\alpha^M\right)^k.
\]
But
\[
\alpha+\overline{\alpha}
=
2.
\]
The combined coefficient is therefore nonzero, which again contradicts the identical vanishing of \eqref{eq:7.1}.

Both possibilities lead to contradictions. Hence \(S_p(d,q)\), and therefore \(c_p(d,q)\), is nonzero for infinitely many odd primes
\(p\).
\end{proof}

The following form will be used repeatedly.

\begin{corollary}
\label{cor:prime-coefficient-support-bound}
Let \(d\) be a positive even integer, and let \(q>0\). Then there exist infinitely many odd
primes \(p\) such that
\[
c_p(d,q)\neq0.
\]
Consequently, for any prescribed bound \(M>0\), one may choose such a
prime satisfying
\(
p>M.
\)
\end{corollary}

\begin{proof}
By Theorem~\ref{thm:prime-coefficient-nonvanishing}, the set of odd
primes \(p\) satisfying
\[
c_p(d,q)\neq0
\]
is infinite. Every infinite set of primes is unbounded, so there exist such primes exceeding any prescribed bound \(M\).
\end{proof}

\subsection{Reduction to a Single Exceptional Factor}
\label{subsec:exceptional-reduction}

We first recall the reduction obtained from the equality
\[
P_\lambda(x)=P_\mu(x).
\]
Recall
\(
P_\lambda(x)
=
E_\lambda(x)O_\lambda(x), 
\)
and
\(
P_\mu(x)
=
E_\mu(x)O_\mu(x).
\)
The Quotient Reduction theorem implies that there exists an integer \(r\ge1\) such that
\[
E_\lambda(x)=(x^2+1)^rE_\mu(x)
\]
and
\[
O_\mu(x)=(x^2+1)^rO_\lambda(x).
\]

\begin{theorem}
\label{thm:odd-quotient-first-step}
Suppose that
\[
O_\mu(x)
=
(x^2+1)^rO_\lambda(x)
\]
for some integer \(r\geq1\). Then
\[
r=1.
\]
Moreover,
\(
s_1(\mu)=2,\text{ and }
s_1(\lambda)\in\{0,1\}.
\)
\end{theorem}

\begin{proof}
Let
\[
a=s_1(\mu),
\qquad
b=s_1(\lambda).
\]

Taking normalized reciprocals gives
\[
\prod_{\substack{d\ge1\\ d\text{ odd}}}
S_{d,s_d(\mu)}(t)
=
(1+t^2)^r
\prod_{\substack{d\ge1\\ d\text{ odd}}}
S_{d,s_d(\lambda)}(t),
\]
where
\[
S_{d,m}(t)
=
t^{dm}\He_m^{[-d]}(t^{-d}).
\]

For \(d=1\),
\[
S_{1,m}(t)
=
1+\binom{m}{2}t^2
+
3\binom{m}{4}t^4
+
O(t^6),
\]
while for every odd integer \(d\ge3\),
\[
S_{d,m}(t)
=
1+O(t^6).
\]
Hence, modulo \(t^6\),
\[
S_{1,a}(t)
\equiv
(1+t^2)^rS_{1,b}(t).
\]

Expanding
\[
(1+t^2)^r
=
1+rt^2+\binom{r}{2}t^4+O(t^6),
\]
and comparing coefficients of \(t^2\) gives
\[
\binom{a}{2}
=
r+\binom{b}{2}.
\tag{7.2}\label{eq:7.2}
\]
Comparing coefficients of \(t^4\) gives
\[
3\binom{a}{4}
=
\binom{r}{2}
+
r\binom{b}{2}
+
3\binom{b}{4}.
\tag{7.3} \label{eq:7.3}
\]

From \eqref{eq:7.2},
\[
r
=
\binom{a}{2}
-
\binom{b}{2}.
\]
Since \(r\ge1\), it follows that
\[
\binom{a}{2}
>
\binom{b}{2},
\]
and hence
\[
a>b.
\]

Substituting
\[
r
=
\binom{a}{2}
-
\binom{b}{2}
\]
into \eqref{eq:7.3} and simplifying yields
\[
(a-b)
\left(
a^2+ab+b^2-3a-3b+2
\right)
=
0.
\]

Since \(a>b\), the first factor is nonzero. Therefore
\[
a^2+ab+b^2-3a-3b+2
=
0.
\tag{7.4}\label{eq:7.4}
\]

We claim that \(a\ge3\) is impossible.

If \(b=0\), then
\[
a^2+ab+b^2-3a-3b+2
=
(a-1)(a-2),
\]
which is strictly positive for every \(a\ge3\).

If \(b\ge1\), then
\[
\begin{aligned}
a^2+ab+b^2-3a-3b+2
&=
a(a-3)+b(a+b-3)+2.
\end{aligned}
\]
Since \(a\ge3\),
\[
a(a-3)\ge0,
\]
and
\[
b(a+b-3)\ge b^2>0.
\]
Hence the right-hand side is strictly positive, contradicting \eqref{eq:7.4}.

Therefore
\[
a=2.
\]

Since \(a>b\), we obtain
\[
b\in\{0,1\}.
\]

Finally,
\[
r
=
\binom{2}{2}
-
\binom{b}{2}
=
1,
\]
because
\[
\binom{0}{2}
=
\binom{1}{2}
=
0.
\]

Thus
\[
r=1,
\]
and
\[
s_1(\mu)=2,
\qquad
s_1(\lambda)\in\{0,1\},
\]
as claimed.
\end{proof}

From Theorem \ref{thm:quotient-reduction} and Theorem \ref{thm:odd-quotient-first-step}, we obtain that when $e_\lambda \neq e_\mu$,
\[
\boxed{
E_\lambda(x)
=
(x^2+1)E_\mu(x)
}
\tag{7.5}\label{eq:7.5}
\]
and
\[
\boxed{
O_\mu(x)
=
(x^2+1)O_\lambda(x).
}
\tag{7.6}\label{eq:7.6}
\]

In particular,
\[
e_\lambda-e_\mu=2
\]
and
\[
o_\mu-o_\lambda=2.
\]

\subsection{The Odd-Part Discrepancy}
\label{subsec:odd-discrepancy-exceptional}

We now determine the residual discrepancy between the odd parts of
\(\lambda\) and \(\mu\).

By Theorem~\ref{thm:odd-quotient-first-step}, the identity
\[
O_\mu(x)
=
(x^2+1)O_\lambda(x)
\]
implies
\[
s_1(\mu)=2,
\qquad
s_1(\lambda)\in\{0,1\}.
\]

We show that, after removing the distinguished pair \((1,1)\) from \(\mu\), the remaining odd subpartitions of \(\lambda\) and \(\mu\) are related by Operation~\(1\).

\begin{proposition}
\label{prop:odd-exceptional-decomposition}
Suppose
\[
O_\mu(x)=(x^2+1)O_\lambda(x).
\]
Then, for every odd integer \(d\ge3\),
\[
\max\{s_d(\lambda),s_d(\mu)\}\ge2
\quad\Longrightarrow\quad
s_d(\lambda)=s_d(\mu).
\]
After cancelling all common odd parts of multiplicity at least two, the odd subpartition of \(\lambda\) is related by a single legal application of Operation~1 to the odd subpartition obtained from \(\mu\) by deleting the two parts \(1,1\).

\end{proposition}

\begin{proof}
We have
\[
s_1(\mu)=2,
\qquad
s_1(\lambda)\in\{0,1\}
\]
from Theorem~\ref{thm:odd-quotient-first-step}.

Write
\[
b=s_1(\lambda)\in\{0,1\}.
\]

Recall that, for every positive odd integer \(d\) and every \(m\ge0\), the normalized reciprocal of
\[
G_{d,m}(x)=\He_m^{[-d]}(x^d)
\]
is
\[
S_{d,m}(t)
=
t^{dm}G_{d,m}(t^{-1}).
\]
By Lemma~\ref{lem:odd-reciprocal-expansion},
\[
S_{d,m}(t)
=
1+\binom{m}{2}d\,t^{2d}
+
O(t^{4d}).
\]

Since
\[
s_1(\mu)=2,
\]
the part-size-\(1\) factor in \(O_\mu\) is
\[
\He_2^{[-1]}(x)=x^2+1,
\]
whose normalized reciprocal is
\[
S_{1,2}(t)=1+t^2.
\]
On the other hand,
\[
b\in\{0,1\},
\]
so
\[
S_{1,b}(t)=1.
\]

Taking normalized reciprocals of
\[
O_\mu(x)
=
(x^2+1)O_\lambda(x)
\]
therefore gives
\[
(1+t^2)
\prod_{\substack{d\ge3\\d\text{ odd}}}
S_{d,s_d(\mu)}(t)
=
(1+t^2)
\prod_{\substack{d\ge3\\d\text{ odd}}}
S_{d,s_d(\lambda)}(t).
\]
Cancelling the common factor \(1+t^2\), we obtain
\[
\prod_{\substack{d\ge3\\d\text{ odd}}}
S_{d,s_d(\mu)}(t)
=
\prod_{\substack{d\ge3\\d\text{ odd}}}
S_{d,s_d(\lambda)}(t).
\tag{7.7}\label{eq:7.7}
\]

We now apply the same least-degree argument used in
Theorem~\ref{thm:recovery-higher-odd}.

Let \(d\ge3\) be odd, and suppose inductively that all odd part sizes
\(e\) with
\[
3\le e<d
\]
have already been dealt with: if either multiplicity is at least two, then the two multiplicities have been shown equal and the corresponding common reciprocal factors have been cancelled; if both multiplicities
belong to \(\{0,1\}\), then both reciprocal factors are equal to \(1\) and contribute nothing.

After these cancellations, the coefficient of \(t^{2d}\) in the
left-hand side of \eqref{eq:7.7} is
\[
\binom{s_d(\mu)}{2}d,
\]
while the coefficient of \(t^{2d}\) in the right-hand side is
\[
\binom{s_d(\lambda)}{2}d.
\]
All factors indexed by an odd integer \(e>d\) begin in degree
\(2e>2d\), and hence do not contribute.

Therefore
\[
\binom{s_d(\lambda)}{2}
=
\binom{s_d(\mu)}{2}.
\]
Since the map
\[
m\longmapsto \binom{m}{2}
\]
is injective for \(m\ge1\), with the only ambiguity
\[
\binom{0}{2}
=
\binom{1}{2}
=
0,
\]
we conclude that
\[
\max\{s_d(\lambda),s_d(\mu)\}\ge2
\quad\Longrightarrow\quad
s_d(\lambda)=s_d(\mu).
\]

Thus every odd part size \(d\ge3\) occurring with multiplicity at least two occurs in both partitions with the same multiplicity. Cancel all of these common Hermite factors from
\[
O_\mu(x)
=
(x^2+1)O_\lambda(x).
\]

The factor \(x^2+1\) on the left is precisely the contribution of the two parts \(1,1\) in \(\mu\), so it cancels the distinguished factor \(x^2+1\) on the right. What remains consists entirely of singleton
odd parts.

Define
\[
A_\lambda
=
\{d\ge3:\ d\text{ is odd and }s_d(\lambda)=1\},
\]
and
\[
A_\mu
=
\{d\ge3:\ d\text{ is odd and }s_d(\mu)=1\}.
\]
Since
\[
s_1(\lambda)=b\in\{0,1\},
\]
the residual polynomial identity becomes
\[
x^b
\prod_{d\in A_\lambda}x^d
=
\prod_{e\in A_\mu}x^e.
\]
Hence
\[
b+\sum_{d\in A_\lambda}d
=
\sum_{e\in A_\mu}e.
\tag{7.8}\label{eq:7.8}
\]

Let
\[
A_\lambda'
=
\begin{cases}
A_\lambda,&b=0,\\[2mm]
A_\lambda\cup\{1\},&b=1.
\end{cases}
\]
Then \eqref{eq:7.8} is equivalent to
\[
\sum_{d\in A_\lambda'}d
=
\sum_{e\in A_\mu}e.
\]

Remove the common elements and set
\[
A=A_\lambda'\setminus A_\mu,
\qquad
B=A_\mu\setminus A_\lambda'.
\]
Then
\[
\sum_{a\in A}a
=
\sum_{b\in B}b.
\tag{7.9}\label{eq:7.9}
\]

Every part in \(A\) occurs in \(\lambda\) with multiplicity exactly one. We claim that every part in \(B\) is absent from \(\lambda\). Indeed, let \(b\in B\). Then
\[
s_b(\mu)=1
\]
and
\[
s_b(\lambda)\neq1.
\]
If \(s_b(\lambda)\ge2\), the higher-multiplicity rigidity proved above
would imply
\[
s_b(\mu)=s_b(\lambda)\ge2,
\]
contradicting \(s_b(\mu)=1\). Therefore
\[
s_b(\lambda)=0.
\]

Thus the distinct odd parts in \(A\) may be removed from \(\lambda\) and the distinct odd parts in \(B\) inserted. By \eqref{eq:7.9}, the total sum of the removed parts equals the total sum of the inserted parts. Hence this is a legal application of Operation~\(1\).

After this operation, the odd subpartition of \(\lambda\) is exactly the odd subpartition of \(\mu\) with the distinguished pair \((1,1)\) deleted.
\end{proof}

\subsection{The Exact Multiplicity of the Part \(2\)}
\label{subsec:exact-part-two}

From the even quotient identity \eqref{eq:7.5}, Proposition
\ref{prop:cyclotomic-parity-exceptional} gives
\[
s_2(\lambda)\equiv0\pmod2,
\qquad
s_2(\mu)\equiv1\pmod2,
\]
and
\[
s_4(\lambda)\equiv1\pmod2,
\qquad
s_4(\mu)\equiv0\pmod2.
\]
For every positive even integer \(d\neq2,4\),
\[
s_d(\lambda)\equiv s_d(\mu)\pmod2.
\tag{$\dagger$} \label{eq:dagger}
\]

Moreover, Corollary~\ref{cor:core-equality-exceptional} gives
\[
C_\lambda(x)=C_\mu(x).
\]

Taking normalized reciprocals and then formal logarithms, we obtain
\[
\sum_{\substack{d\geq2\\d\text{ even}}}
\log T_{d,s_d(\lambda)}(t)
=
\sum_{\substack{d\geq2\\d\text{ even}}}
\log T_{d,s_d(\mu)}(t).
\tag{7.10}\label{eq:7.10}
\]

By Lemma~\ref{lem:logarithmic-divisor-support}, the coefficient of
\(t^N\) in \eqref{eq:7.10} receives contributions only from those positive even
integers \(d\) satisfying
\[
d\mid N.
\]

We now prove that the part \(2\) occurs exactly once in \(\mu\) and is
absent from \(\lambda\).

\begin{proposition}
\label{prop:exact-two-multiplicity}
We have
\(
s_2(\lambda)=0
\; \text{and}\;
s_2(\mu)=1.
\)
\end{proposition}

\begin{proof}
By Proposition~\ref{prop:cyclotomic-parity-exceptional},
\[
s_2(\lambda)\equiv0\pmod2,
\qquad
s_2(\mu)\equiv1\pmod2.
\]
Hence there exist integers \(a,b\ge0\) such that
\[
s_2(\lambda)=2a,
\qquad
s_2(\mu)=2b+1.
\]

We first show that \(a=b\).
By Corollary~\ref{cor:core-equality-exceptional},
\[
C_\lambda(x)=C_\mu(x),
\]
and hence Proposition~\ref{prop:core-reciprocal-product} gives
\[
\prod_{\substack{d\ge2\\ d\text{\rm\ even}}}
T_{d,s_d(\lambda)}(t)
=
\prod_{\substack{d\ge2\\ d\text{\rm\ even}}}
T_{d,s_d(\mu)}(t).
\tag{7.11}\label{eq:7.11}
\]

By Lemma~\ref{lem:core-reciprocal-first-term},
\[
T_{2,2a}(t)
=
1-2a\,t^2+O(t^4),
\]
whereas
\[
T_{2,2b+1}(t)
=
1-2b\,t^2+O(t^4).
\]
For every even integer \(d>2\), Lemma~\ref{lem:core-reciprocal-first-term} gives
\[
T_{d,m}(t)=1+O(t^d),
\]
and hence no factor with index \(d>2\) contributes to the coefficient
of \(t^2\) in~\eqref{eq:7.11}. Comparing the coefficients of \(t^2\) therefore
gives
\[
-2a=-2b,
\]
so
\[
a=b.
\]

Thus there exists an integer \(q\ge0\) such that
\[
s_2(\lambda)=2q,
\qquad
s_2(\mu)=2q+1.
\]

Suppose, for a contradiction, that \(q>0\).
Let \(N\) be larger than every part occurring in either
\(\lambda\) or \(\mu\). By Corollary~\ref{cor:prime-coefficient-support-bound}, there exists an odd prime
\(p\) such that
\[
2p>N
\]
and
\[
c_p(2,q)\neq0.
\]

By Corollary~\ref{cor:logarithmic-core-identity},
\[
\sum_{\substack{d\ge2\\ d\text{\rm\ even}}}
\log T_{d,s_d(\lambda)}(t)
=
\sum_{\substack{d\ge2\\ d\text{\rm\ even}}}
\log T_{d,s_d(\mu)}(t).
\tag{7.12}\label{eq:7.12}
\]

We compare the coefficient of \(t^{2p}\) in~\eqref{eq:7.12}.
By Lemma~\ref{lem:logarithmic-divisor-support}, a part size \(d\) can contribute to this coefficient
only if
\[
d\mid2p.
\]
Since \(p\) is odd, the only positive even divisors of \(2p\) are
\[
2
\qquad\text{and}\qquad
2p.
\]

Because \(2p>N\), neither partition contains a part of size \(2p\).
Thus the factors corresponding to part size-\(2p\) in~\eqref{eq:7.12} are both zero. Consequently,
the entire coefficient of \(t^{2p}\) comes from the part-size-\(2\)
terms.

Writing
\[
u=t^2,
\]
the difference of the part-size-\(2\) terms is
\[
\log T_{2,2q}(u)-\log T_{2,2q+1}(u)
=
-\log\frac{T_{2,2q+1}(u)}{T_{2,2q}(u)}.
\]
By the definition of \(c_p(2,q)\), the coefficient of \(u^p\) in
this expression is
\[
-c_p(2,q),
\]
which is nonzero by our choice of \(p\).

Therefore the coefficient of \(t^{2p}\) in the difference between the
two sides of~\eqref{eq:7.12} is equal to
\[
-c_p(2,q)\neq0.
\]
This is impossible, since the two sides of~\eqref{eq:7.12} are equal as formal
power series, and hence every coefficient of their difference must
vanish. Therefore
\[
q=0.
\]
Hence
\[
s_2(\lambda)=0
\qquad\text{and}\qquad
s_2(\mu)=1.
\]

In particular,
\[
T_{2,s_2(\lambda)}(t)=T_{2,0}(t)=1
\]
and
\[
T_{2,s_2(\mu)}(t)=T_{2,1}(t)=1. \qedhere
\]
\end{proof}

\subsection{The Exact Multiplicity of the Part \(4\)}
\label{subsec:exact-part-four}

We next prove that the part \(4\) occurs exactly once in \(\lambda\) and
is absent from \(\mu\).


\begin{proposition}
\label{prop:exact-four-multiplicity}
We have
\[
s_4(\lambda)=1
\qquad\text{and}\qquad
s_4(\mu)=0.
\]
\end{proposition}

\begin{proof}
By Proposition~\ref{prop:cyclotomic-parity-exceptional},
\[
s_4(\lambda)\equiv1\pmod2,
\qquad
s_4(\mu)\equiv0\pmod2.
\]
Hence there exist integers \(a,b\ge0\) such that
\[
s_4(\lambda)=2a+1,
\qquad
s_4(\mu)=2b.
\]

We first show that \(a=b\).
By Corollary~\ref{cor:core-equality-exceptional},
\[
C_\lambda(x)=C_\mu(x),
\]
and hence Proposition~\ref{prop:core-reciprocal-product} gives
\[
\prod_{\substack{d\ge2\\ d\text{\rm even}}}
T_{d,s_d(\lambda)}(t)
=
\prod_{\substack{d\ge2\\ d\text{\rm even}}}
T_{d,s_d(\mu)}(t).
\tag{7.13}\label{eq:7.13}
\]

By Proposition~\ref{prop:exact-two-multiplicity},
\[
T_{2,s_2(\lambda)}(t)=T_{2,0}(t)=1,
\qquad
T_{2,s_2(\mu)}(t)=T_{2,1}(t)=1,
\]
so the part-size-\(2\) factors contribute nothing.

By Lemma~\ref{lem:core-reciprocal-first-term},
\[
T_{4,2a+1}(t)
=
1-2a\,t^4+O(t^8),
\]
whereas
\[
T_{4,2b}(t)
=
1-2b\,t^4+O(t^8).
\]

For every even integer \(d>4\),
\[
T_{d,m}(t)=1+O(t^d),
\]
and since \(d\ge6\), no such factor contributes to the coefficient of
\(t^4\) in~\eqref{eq:7.13}.

Comparing the coefficients of \(t^4\) therefore gives
\[
-2a=-2b,
\]
so
\[
a=b.
\]

Thus there exists an integer \(q\ge0\) such that
\[
s_4(\lambda)=2q+1,
\qquad
s_4(\mu)=2q.
\]

Suppose, for a contradiction, that \(q>0\).
Let \(N\) be larger than every part occurring in either partition.
By Corollary~\ref{cor:prime-coefficient-support-bound}, there exists an odd prime \(p\) such that
\[
2p>N
\]
and
\[
c_p(4,q)\neq0.
\]

By Corollary~\ref{cor:logarithmic-core-identity},
\[
\sum_{\substack{d\ge2\\ d\text{\rm even}}}
\log T_{d,s_d(\lambda)}(t)
=
\sum_{\substack{d\ge2\\ d\text{\rm even}}}
\log T_{d,s_d(\mu)}(t).
\tag{7.14}\label{eq:7.14}
\]

We compare the coefficient of \(t^{4p}\) in~\eqref{eq:7.14}.
By Lemma~\ref{lem:logarithmic-divisor-support}, a part size \(d\) contributes only if
\[
d\mid4p.
\]
Since \(p\) is odd, the positive even divisors of \(4p\) are
\[
2,\ 4,\ 2p,\ 4p.
\]

The part-size-\(2\) terms vanish because
\[
T_{2,s_2(\lambda)}=T_{2,0}=1,
\qquad
T_{2,s_2(\mu)}=T_{2,1}=1.
\]
Moreover,
\[
2p>N,
\]
so neither partition contains a part of size \(2p\) or \(4p\).
Hence the only contribution to the coefficient of \(t^{4p}\) comes
from the part-size-\(4\) terms.

Writing \(u=t^4\),
\[
\log T_{4,2q}(u)
-
\log T_{4,2q+1}(u)
=
-
\log
\frac{T_{4,2q+1}(u)}
     {T_{4,2q}(u)}.
\]
By the definition of \(c_p(4,q)\), the coefficient of \(u^p\) in this
expression is
\[
-c_p(4,q),
\]
which is nonzero.

Therefore the coefficient of \(t^{4p}\) in the difference between the
two sides of~\eqref{eq:7.14} is nonzero, contradicting the fact that the two sides
are equal as formal power series.

Hence \(q=0\), and therefore
\[
s_4(\lambda)=1,
\qquad
s_4(\mu)=0.
\]

In particular,
\[
T_{4,s_4(\lambda)}(t)=T_{4,1}(t)=1,
\qquad
T_{4,s_4(\mu)}(t)=T_{4,0}(t)=1.\qedhere
\]
\end{proof}


\subsection{Equality of the Remaining Even Multiplicities}
\label{subsec:remaining-even-multiplicities}

We now recover every even multiplicity of size at least \(6\).

\begin{proposition}
\label{prop:remaining-even-multiplicities}
For every positive even integer \(d\geq6\),
\[
s_d(\lambda)=s_d(\mu).
\]
\end{proposition}

\begin{proof}
By \eqref{eq:dagger}, the two multiplicities have the same parity:
\[
\varepsilon(s_d(\lambda))
=
\varepsilon(s_d(\mu)).
\tag{7.15} \label{eq:7.15}
\]
From Propositions \ref{prop:exact-two-multiplicity} and \ref{prop:exact-four-multiplicity}, the reciprocal core identity becomes
\[
\prod_{\substack{d\geq6\\d\text{ even}}}
T_{d,s_d(\lambda)}(t)
=
\prod_{\substack{d\geq6\\d\text{ even}}}
T_{d,s_d(\mu)}(t).
\tag{7.16}\label{eq:7.16}
\]

We proceed through the positive even integers \(d\geq6\) in increasing
order.

Assume inductively that
\[
s_e(\lambda)=s_e(\mu)
\]
for every even integer \(e\) satisfying
\[
6\leq e<d.
\]
Canceling the corresponding common factors from \eqref{eq:7.16}, we obtain
\[
\prod_{\substack{e\geq d\\e\text{ even}}}
T_{e,s_e(\lambda)}(t)
=
\prod_{\substack{e\geq d\\e\text{ even}}}
T_{e,s_e(\mu)}(t).
\tag{7.17}\label{eq:7.17}
\]

By Lemma~\ref{lem:core-reciprocal-first-term},
\[
T_{d,s_d(\lambda)}(t)
=
1-
\left(
s_d(\lambda)
-
\varepsilon(s_d(\lambda))
\right)t^d
+
O(t^{2d}),
\]
and similarly for \(\mu\).

Every factor indexed by \(e>d\) has its first possible nonconstant term
in degree \(e>d\). Hence no factor with index larger than \(d\)
contributes to the coefficient of \(t^d\) in \eqref{eq:7.17}.

Equality of the \(t^d\)-coefficients therefore gives
\[
s_d(\lambda)
-
\varepsilon(s_d(\lambda))
=
s_d(\mu)
-
\varepsilon(s_d(\mu)).
\]

Using the parity equality \eqref{eq:7.15}, we conclude that
\[
s_d(\lambda)=s_d(\mu).
\]

This completes the induction.
\end{proof}

Combining Propositions~\ref{prop:exact-two-multiplicity},
\ref{prop:exact-four-multiplicity}, and
\ref{prop:remaining-even-multiplicities}, we have proved
\[
\boxed{
\begin{aligned}
s_2(\lambda)&=0,
&
s_2(\mu)&=1,\\
s_4(\lambda)&=1,
&
s_4(\mu)&=0,
\end{aligned}
}
\]
and
\[
\boxed{
s_d(\lambda)=s_d(\mu)
\qquad
(d\geq6,\ d\text{ even}).
}
\]

\subsection{Construction of the Residual Partitions}
\label{subsec:construct-alpha-beta}

We now construct the residual partitions \(\alpha\) and \(\beta\).

By Proposition~\ref{prop:exact-two-multiplicity},
\[
s_2(\lambda)=0,
\qquad
s_2(\mu)=1,
\]
and by Proposition~\ref{prop:exact-four-multiplicity},
\[
s_4(\lambda)=1,
\qquad
s_4(\mu)=0.
\]
Furthermore, Proposition~\ref{prop:remaining-even-multiplicities}
shows that
\[
s_d(\lambda)=s_d(\mu)
\]
for every even integer \(d\ge6\).

Define
\[
\alpha=\lambda\setminus(4),
\qquad
\beta=\mu\setminus(2,1,1).
\]
Equivalently,
\[
\lambda=\alpha\sqcup(4),
\qquad
\mu=\beta\sqcup(2,1,1).
\]

The even subpartitions of \(\alpha\) and \(\beta\) are therefore identical.

On the other hand, Proposition~\ref{prop:odd-exceptional-decomposition} shows that, after removing the distinguished pair of parts \((1,1)\) from \(\mu\), the remaining odd parts of \(\lambda\) and \(\mu\) are related by a single legal application of Operation~1. Since removing the part \(4\) from \(\lambda\) and the part \(2\) from \(\mu\) does not affect their odd subpartitions, it follows that the odd subpartitions of \(\alpha\) and \(\beta\) are related by Operation~1.

Hence a legal application of Operation~1 transforms \(\alpha\) into \(\beta\). Since Operation~1 changes only odd parts, the same operation may be performed while retaining the distinguished part \(4\). Thus
\[
\alpha\sqcup(4)
\;\sim\;
\beta\sqcup(4).
\]

We now apply Operation~2 to \(\beta\sqcup(4)\). By construction,
\[
\beta=\mu\setminus(2,1,1).
\]
Since
\[
s_1(\mu)=2,\qquad
s_2(\mu)=1,\qquad
s_4(\mu)=0,
\]
we have
\[
s_1(\beta)=s_2(\beta)=s_4(\beta)=0.
\]
Therefore the partition \(\beta\sqcup(4)\) contains exactly one part of size \(4\), no parts of size \(2\), and no parts of size \(1\). Consequently, Operation~2 is legal and replaces the distinguished part
\(4\) by the parts \((2,1,1)\).

Thus
\[
\lambda
=
\alpha\sqcup(4)
\overset{\mathrm{Op.\,1}}{\sim}
\beta\sqcup(4)
\overset{\mathrm{Op.\,2}}{\sim}
\beta\sqcup(2,1,1)
=
\mu.
\]

Therefore
\[
\lambda\sim\mu.
\]

\subsection{Conclusion of the Exceptional Case}
\label{subsec:exceptional-conclusion}

We summarize the result of this section.

\begin{theorem}[The exceptional case]
\label{thm:exceptional-case}
Let \(\lambda,\mu\vdash n\). Suppose that
\(
e_\lambda\neq e_\mu
\)
and
\(
P_\lambda(x)=P_\mu(x).
\)
Then
\[
\lambda\sim\mu.
\]
More precisely, after possibly interchanging \(\lambda\) and \(\mu\),
there exist partitions \(\alpha,\beta\vdash n-4\) such that
\[
\lambda=\alpha\sqcup(4),
\qquad
\mu=\beta\sqcup(2,1,1),
\]
and \(\alpha\) and \(\beta\) are related by Operation~\(1\).
\end{theorem}

\begin{proof}
After interchanging the partitions if necessary, the quotient reduction
gives
\(
E_\lambda=(x^2+1)E_\mu
\)
and
\(
O_\mu=(x^2+1)O_\lambda.
\)
The odd quotient analysis gives the decomposition of the odd parts in
Proposition~\ref{prop:odd-exceptional-decomposition}.

The cyclotomic comparison determines the parity of every even-part multiplicity. The prime-coefficient nonvanishing theorem, applied first at part size \(2\) and then at part size \(4\), gives
\[
s_2(\lambda)=0,
\qquad
s_2(\mu)=1,
\]
and
\[
s_4(\lambda)=1,
\qquad
s_4(\mu)=0.
\]

Proposition~\ref{prop:remaining-even-multiplicities} shows that all remaining even multiplicities agree. Hence there exist partitions
\(\alpha,\beta\vdash n-4\) satisfying
\[
\lambda=\alpha\sqcup(4),
\qquad
\mu=\beta\sqcup(2,1,1),
\]
with the even subpartitions of \(\alpha\) and \(\beta\) identical and their odd subpartitions related by Operation~\(1\).

Thus
\[
\alpha\sim\beta,
\]
and one legal application of Operation~\(2\) replaces the distinguished block \((4)\) by \((2,1,1)\). Therefore
\[
\lambda\sim\mu. \qedhere
\]
\end{proof}

\begin{theorem}[Classification of collisions]
\label{thm:main-classification}
Let \(\lambda\) and \(\mu\) be partitions. Then
\[
P_\lambda(x)=P_\mu(x)
\]
if and only if \(\lambda\sim\mu\), where \(\sim\) is the equivalence relation generated by Operations~1 and~2.
\end{theorem}

\begin{proof}
Suppose first that
\[
P_\lambda(x)=P_\mu(x).
\]
If \(e_\lambda=e_\mu\), then \(\lambda\sim\mu\) by
Theorem~\ref{thm:equal-even-weight-classification}. Otherwise \(e_\lambda\neq e_\mu\), and
Theorem~\ref{thm:exceptional-case} gives \(\lambda\sim\mu\).
Hence
\[
P_\lambda(x)=P_\mu(x)
\quad\Longrightarrow\quad
\lambda\sim\mu.
\]

The converse was proved in Corollary \ref{cor:equivalence-preserves}.
This completes the proof.
\end{proof}

\section{Partition Statistics from Special Roots}
\label{sec:selected-roots}
In this final section, we illustrate how the factorization of the twisted Foulkes character polynomial immediately yields information about the underlying partition through the multiplicities of certain distinguished roots. These observations provide a simple geometric interpretation of the polynomial: the orders of vanishing at \(0\), at \(1\), and at roots of unity record natural combinatorial statistics of the partition.

\begin{proposition}
\label{prop:special-root-multiplicities}
Let
\[
\nu=(1^{s_1},2^{s_2},\ldots,n^{s_n})\vdash n.
\]
Then the following hold.
\begin{enumerate}[label=\rm(\roman*)]
\item
The order of vanishing of \(P_\nu(x)\) at \(x=0\) is
\[
\operatorname{ord}_{x=0} P_\nu(x)
=
\sum_{\substack{d\ge1\\ d\text{\rm\ odd}\\ s_d\text{\rm\ odd}}} d.
\]

\item
The order of vanishing of \(P_\nu(x)\) at \(x=1\) is
\[
\operatorname{ord}_{x=1} P_\nu(x)
=
\#\left\{
d\ge2:
d\text{\rm\ even},\;
s_d\text{\rm\ odd}
\right\}.
\]

\item
Let \(\zeta\) be a root of unity of order \(N\neq4\). Then
\[
\operatorname{ord}_{x=\zeta} P_\nu(x)
=
\#\left\{
d\ge2:
d\text{\rm\ even},\;
N\mid d,\;
s_d\text{\rm\ odd}
\right\}.
\]
\end{enumerate}
\end{proposition}

\begin{proof}
Recall that
\[
P_\nu(x)
=
\prod_{\substack{d\ge1\\ d\text{\rm\ even}}}
F_{d,s_d}(x)
\prod_{\substack{d\ge1\\ d\text{\rm\ odd}}}
G_{d,s_d}(x),
\]
where
\[
F_{d,m}(x)=\He^{[-d]}_m(1-x^d),
\qquad
G_{d,m}(x)=\He^{[-d]}_m(x^d).
\]

For~\rm(i), consider first an odd integer \(d\). Since
\[
G_{d,s_d}(x)
=
\He^{[-d]}_{s_d}(x^d),
\]
the factor \(G_{d,s_d}\) vanishes at \(x=0\) if and only if
\(\He^{[-d]}_{s_d}(0)=0\), which occurs precisely when \(s_d\) is odd.
In that case \(0\) is a simple root of
\(\He^{[-d]}_{s_d}\), and the substitution \(y=x^d\) gives a zero of
multiplicity \(d\) at \(x=0\). On the other hand, for even \(d\),
\[
F_{d,s_d}(0)=\He^{[-d]}_{s_d}(1)\neq0,
\]
since all roots of \(\He^{[-d]}_{s_d}\) are purely imaginary.
Summing the contributions of the odd factors gives
\[
\operatorname{ord}_{x=0} P_\nu(x)
=
\sum_{\substack{d\ge1\\ d\text{\rm\ odd}\\ s_d\text{\rm\ odd}}} d.
\]

For~\rm(ii), if \(d\) is even, then
\[
F_{d,s_d}(1)=\He^{[-d]}_{s_d}(0),
\]
so \(F_{d,s_d}(1)=0\) if and only if \(s_d\) is odd. By Lemma~\ref{lem:unit-circle-roots-even-factor},
every such zero at \(x=1\) is simple. If \(d\) is odd, then
\[
G_{d,s_d}(1)=\He^{[-d]}_{s_d}(1)\neq0.
\]
Therefore the multiplicity of \(x=1\) as a root of \(P_\nu\) is exactly
the number of even part sizes occurring with odd multiplicity, namely
\[
\operatorname{ord}_{x=1} P_\nu(x)
=
\#\left\{
d\ge2:
d\text{\rm\ even},\;
s_d\text{\rm\ odd}
\right\}.
\]

For~\rm(iii), let \(\zeta\) have order \(N\neq4\). By Lemma~\ref{lem:unit-circle-roots-even-factor}, for an
even integer \(d\),
\[
F_{d,s_d}(\zeta)=0
\]
if and only if \(s_d\) is odd and
\[
\zeta^d=1.
\]
Since \(\zeta\) has order \(N\), this is equivalent to
\[
N\mid d.
\]
Moreover, every such zero is simple.

It remains to show that none of the odd Hermite factors vanishes at \(\zeta\). Suppose
\[
G_{d,s_d}(\zeta)=0
\]
for some odd \(d\). Then
\[
\He^{[-d]}_{s_d}(\zeta^d)=0,
\]
so by Theorem~\ref{thm:imaginary-rescaling} there is a real root \(\rho\) of the ordinary Hermite
polynomial \(\He_{s_d}\) such that
\[
\zeta^d=i\sqrt d\,\rho.
\]
Taking absolute values yields
\[
d\rho^2=1.
\]
Since \(\rho\) is an algebraic integer, so is \(\rho^2=1/d\). As
\(1/d\in\mathbb Q\), this forces \(1/d\in\mathbb Z\), and hence \(d=1\).
Thus \(\rho=\pm1\), so
\[
\zeta=\pm i,
\]
which has order \(4\), contrary to \(N\neq4\).

Hence only the even factors contribute zeros at \(\zeta\), and each
contributes with multiplicity one precisely when \(N\mid d\) and
\(s_d\) is odd. Therefore
\[
\operatorname{ord}_{x=\zeta} P_\nu(x)
=
\#\left\{
d\ge2:
d\text{\rm\ even},\;
N\mid d,\;
s_d\text{\rm\ odd}
\right\}.\qedhere
\]
\end{proof}

These formulas complement the classification theorem by showing that the twisted Foulkes character polynomial records natural combinatorial data of the underlying partition directly through the multiplicities
of distinguished roots. Thus the partition information encoded by $P_\nu(x)$ is visible not only through its global factorization, but also through its local behavior at special points of the complex plane.

\bibliographystyle{plain}
\bibliography{bibliography}

\end{document}